\documentclass[12pt]{amsart}
\usepackage{amsmath, amssymb, amsfonts}
\usepackage[all]{xypic}
\usepackage[pdftex]{graphicx}

\theoremstyle{plain}
\newtheorem{theorem}{Theorem}[section]
\newtheorem{lemma}[theorem]{Lemma}
\newtheorem{corollary}[theorem]{Corollary}
\newtheorem{prop}[theorem]{Proposition}

\newtheorem{conj}[theorem]{Conjecture}

\theoremstyle{remark}

\newtheorem{remark}[theorem]{Remark}

\newtheorem{example}[theorem]{Example}

\newtheorem*{note*}{Note}
\newtheorem*{remark*}{Remark}
\newtheorem*{example*}{Example}

\theoremstyle{definition}
\newtheorem*{definition*}{Definition}
\newtheorem{definition}[theorem]{Definition}

\newcommand{\Z}{\mathbb{Z}}
\newcommand{\R}{\mathbb{R}}
\newcommand{\Q}{\mathbb{Q}}
\newcommand{\C}{\mathbb{C}}
\newcommand{\N}{\mathbb{N}}
\newcommand{\F}{\mathbb{F}}
\newcommand{\Aff}{\mathrm{Aff}}

\newcommand{\Ann}{\mathrm{Ann}}
\newcommand{\Aut}{\mathrm{Aut}}
\newcommand{\Gal}{\mathrm{Gal}}

\newcommand{\cl}{\mathrm{cl}}

\newcommand{\nr}{\mathrm{nr}}

\newcommand{\Hom}{\mathrm{Hom}}

\newcommand{\Fit}{\mathrm{Fit}}

\newcommand{\res}{\mathrm{res}}
\newcommand{\quot}{\mathrm{quot}}
\newcommand{\proj}{\mathrm{proj}}
\newcommand{\tors}{\mathrm{tors}}
\newcommand{\Irr}{\mathrm{Irr}}
\newcommand{\ind}{\mathrm{ind}}
\newcommand{\PMod}{\mathrm{PMod}}
\newcommand{\ETNC}{\mathrm{ETNC}}
\newcommand{\SSC}{\mathrm{SSC}}
\newcommand{\LTC}{\mathrm{LTC}}
\newcommand{\loc}{\mathrm{loc}}
\newcommand{\Spec}{\mathrm{Spec}}

\numberwithin{equation}{section}

 \newcommand{\onto}{\twoheadrightarrow}

 \newcommand{\pr}{\mathrm{pr}}

 \newcommand{\mal}{^{\times}}
 \newcommand{\et}{\mathrm{\acute{e}t}}

\title[On the ETNC for Tate motives and unconditional annihilation results]{On the equivariant Tamagawa number conjecture for Tate motives and unconditional annihilation results}

\author{Henri Johnston}
\address{
Department of Mathematics\\
University of Exeter\\
Exeter\\
EX4 4QF\\
U.K.
}
\email{H.Johnston@exeter.ac.uk}
\urladdr{http://emps.exeter.ac.uk/mathematics/staff/hj241}

\author{Andreas Nickel}
\address{
Universit\"{a}t Bielefeld\\
Fakult\"{a}t f\"{u}r Mathematik\\
Postfach 100131\\
Universit\"{a}tsstr. 25\\
33501 Bielefeld\\
Germany}
\email{anickel3@math.uni-bielefeld.de}
\urladdr{http://www.math.uni-bielefeld.de/$\sim$anickel3/english.html}

\date{Version of 7th September 2014}
\begin{document}

\maketitle

\begin{abstract}
Let $L/K$ be a finite Galois extension of number fields with Galois group $G$.
Let $p$ be a prime and let $r \leq 0$ be an integer.
By examining the structure of the $p$-adic group ring $\Z_{p}[G]$,
we prove many new cases of the $p$-part of the
equivariant Tamagawa number conjecture (ETNC) for the pair $(h^{0}(\Spec(L))(r),\Z[G])$.
The same methods can also be applied to other conjectures concerning the vanishing
of certain elements in relative algebraic $K$-groups.
We then prove a conjecture of Burns concerning the annihilation of class groups as
Galois modules for a large class of interesting extensions, including cases in which the full ETNC is not known.
Similarly, we construct annihilators
of higher dimensional algebraic $K$-groups of the ring of integers in $L$.
\end{abstract}

\section{Introduction}

Building on work of Bloch and Kato \cite{MR1086888}, Fontaine and Perrin-Riou \cite{MR1265546}, and Kato \cite{MR1338860},
Burns and Flach \cite{MR1884523} formulated
the equivariant Tamagawa number conjecture (ETNC)
for any motive over $\Q$ with the action of a semisimple $\Q$-algebra, 
describing the leading term at $s=0$ of an equivariant motivic $L$-function in terms of certain 
cohomological Euler characteristics. 
This is a powerful and unifying formulation which, in particular, 
recovers the generalised Stark conjectures and the Birch and Swinnerton-Dyer conjecture.
We refer the reader to the survey article \cite{MR2088713} for a more detailed overview.

In the present article we shall consider the ETNC when specialised to the case of Tate motives
(in principle, our techniques also apply to other cases).
Let $L/K$ be a finite Galois extension of number fields with Galois group $G$ and let $r$ be an integer.
When specialised to the pair $(h^0(\mathrm{Spec}(L))(r),\Z[G])$, 
the ETNC asserts the vanishing of a certain element
$T\Omega(\Q(r)_{L},\Z[G])$ in the relative algebraic $K$-group $K_{0}(\Z[G],\R)$.
This element relates the leading terms at $s=r$ of Artin $L$-functions to
certain natural arithmetic invariants, and when it vanishes we shall say that `$\ETNC(L/K,r)$ holds'.
In the case that $T\Omega(\Q(r)_{L},\Z[G])$ belongs to $K_{0}(\Z[G],\Q)$
(when $r=0$ this is equivalent to Stark's conjecture for $L/K$), the
conjecture breaks down into local conjectures at each prime $p$
thanks to the canonical isomorphism
\begin{equation*}\label{eq:intro-decomp}
K_{0}(\Z[G], \Q) \simeq \bigoplus_{p} K_{0}(\Z_{p}[G], \Q_{p}).
\end{equation*}
We shall say that `$\ETNC_{p}(L/K,r)$ holds' when the $p$-part of $\ETNC(L/K,r)$ holds.

Let $p$ be a prime and let $DT(\Z_{p}[G])$ denote the torsion subgroup of $K_{0}(\Z_{p}[G],\Q_{p})$.
If $N$ is a normal subgroup of $G$, then there is a natural map
\begin{equation}\label{eq:intro-tors-quot-isom}
\quot_{G/N}^{G}:DT(\Z_{p}[G]) \longrightarrow DT(\Z_{p}[G/N]).
\end{equation}
By studying the structure of the $p$-adic group ring $\Z_{p}[G]$,
we shall give criteria for this quotient map to be an isomorphism.
Thus if these criteria are satisfied and $T\Omega(\Q(r)_{L},\Z[G])$ is torsion,
then the functorial properties of the conjecture
show that $\ETNC_{p}(L/K,r)$ is equivalent to $\ETNC_{p}(L^{N}/K,r)$
(here $L^{N}$ denotes the subfield of $L$ fixed by $N$).
Therefore we can prove many new cases of the $p$-part of the ETNC by reducing to known cases of the ETNC and its weaker variants.

We now consider the following concrete example. 
Let $q=\ell^{n}$ be a prime power and let $\Aff(q)$ denote the group of affine transformations on $\F_{q}$, 
the finite field with $q$ elements.
Hence $\Aff(q)$ is isomorphic to the semidirect product $\F_{q} \rtimes \F_{q}^{\times}$ with the natural action.
Note that $\Aff(3) \simeq S_{3}$, the symmetric group on three letters, and
$\Aff(4) \simeq A_{4}$, the alternating group on four letters.
Now let $L/\Q$ be any Galois extension such that the Galois group $G$ is isomorphic to $\Aff(q)$. 
As all complex irreducible characters of $\Aff(q)$ are either linear or rational-valued, we know
by results of Ritter and Weiss  \cite{MR1423032} and of 
Tate \cite[Chapter II, Theorem 6.8]{MR782485}, respectively,
that the strong Stark conjecture (as formulated by Chinburg \cite[Conjecture 2.2]{MR724009}) holds for the extension $L/\Q$.
This is equivalent to the assertion that $T\Omega(\Q(0)_{L},\Z[G])$ is torsion.
Now let $N \simeq \F_{q}$ be the commutator subgroup of $G$.
For every prime $p \neq \ell$, we shall prove that the group ring $\Z_{p}[G]$ is isomorphic to the direct sum
of $\Z_{p}[G/N]$ and some maximal $\Z_{p}$-order.
From this we deduce that the map \eqref{eq:intro-tors-quot-isom} is an isomorphism in this case.
Thus, by the functorial properties of the conjecture, $\ETNC_{p}(L/\Q,0)$ is equivalent to $\ETNC_{p}(L^{N}/\Q,0)$.
However, the extension $L^{N}/\Q$ is abelian and the ETNC is known for all such extensions by work of Burns and Greither \cite{MR1992015} and Flach \cite{MR2863902}.
Therefore for every prime $p \neq \ell$ we can prove $\ETNC_{p}(L/\Q,0)$.
We remark that up until now there has been no known example
of a finite non-abelian group $G$ and an odd prime $p$ dividing the order of $G$
such that $\ETNC_{p}(L/\Q,0)$ has been shown to hold for every extension
$L/\Q$ with $\Gal(L/\Q) \simeq G$.

From the above example regarding $\Aff(q)$ we deduce the following result.
Fix a natural number $n$. We can construct an infinite family of Galois extensions
of number fields $L/F$ with $\Gal(L/F) \simeq C_{n}$ (the cyclic group of order $n$) and
$F/\Q$ non-abelian (indeed non-Galois) such that $\ETNC(L/F,0)$ holds.
To date the only examples $L/F$ with $F/\Q$ non-abelian for which $\ETNC(L/F,0)$ is known to hold have been either trivial, quadratic, or cubic (see \S \ref{subsec:known-cases}).

Now assume that $r<0$ is odd and consider finite Galois extensions $L/K$ of totally real number fields.
Combining the above approach with a recent result of Burns \cite{burns-mc}
allows us to prove even more than in the case $r=0$.
For example, suppose that $K=\Q$ and $\Gal(L/\Q) \simeq \Aff(q)$ as above.
In this case we prove $\ETNC(L/\Q,r)$ outside the $2$-part.
Assuming the conjectural vanishing of certain $\mu$-invariants,
this result was already established by Burns \cite{burns-mc}.
However, we stress that all of our results are unconditional and do not rely on
any conjecture on the vanishing of $\mu$-invariants.

Let $n$ be an odd natural number and let $p$ be an odd prime. Then the dihedral group $D_{2n}$ has a unique subgroup
isomorphic to the cyclic group $C_{n}$ and Breuning  \cite[Proposition 3.2(2)]{MR2031413} showed that the restriction
map
\[
\res^{D_{2n}}_{C_{n}}: DT(\Z_{p}[D_{2n}]) \longrightarrow DT(\Z_{p}[C_{n}])
\]
is injective. Combining this with results of Bley \cite[Corollary 4.3]{MR2226270} on $\ETNC_{p}(L/K,0)$
where $L/K$ is an abelian extension of an imaginary quadratic field, we establish many new cases of
 $\ETNC_{p}(L/\Q,0)$ where $\Gal(L/\Q) \simeq D_{2n}$. In particular, we give an explicit infinite family of finite non-abelian groups
with the property that for each member $G$ there are infinitely many extensions $L/\Q$ with $\Gal(L/\Q) \simeq G$ such that $\ETNC(L/\Q,0)$ holds.
Up until now, the only known family of finite non-abelian groups with this property has been that containing the single group $Q_{8}$,
the quaternion group of order $8$
(this is a result of Burns and Flach \cite[Theorem 4.1]{MR1981031}). Using work of Johnson-Leung (\cite[Main Theorem]{Johnson-ETNC}), 
by the same method we also establish new results for dihedral extensions of $\Q$ in the case $r<0$.

We can also prove certain cases of other conjectures concerning the vanishing of certain elements in relative algebraic $K$-groups, provided that these elements satisfy the appropriate functorial properties.
In particular, we consider the global equivariant epsilon constant conjecture of Bley and Burns \cite{MR2005875},
the local equivariant epsilon constant conjecture of Breuning \cite{MR2078894}, and the
leading term conjecture at $s=1$ of Breuning and Burns \cite[\S 3]{MR2371375}.

We now consider certain annihilation conjectures.
Let $L/K$ be a finite Galois extension of number fields with Galois group $G$.
To each finite set $S$ of places of $K$ containing all archimedean places,
one can associate a so-called `Stickelberger element'
$\theta_{S}$ in the centre of the complex group algebra $\C[G]$.
This Stickelberger element is defined via $L$-values at zero of $S$-truncated Artin $L$-functions attached to the (complex) characters of $G$.
Let us denote the roots of unity of $L$ by $\mu_{L}$ and the class group of $L$ by $\cl_{L}$.
Assume that $S$ contains all finite primes of $K$ that ramify in $L/K$. Then it was independently
shown in \cite{MR525346}, \cite{MR524276} and \cite{MR579702} that when $G$ is abelian we have
\begin{equation}\label{eq:abelian-integrality}
    \Ann_{\Z[G]} (\mu_{L}) \theta_{S} \subseteq \Z[G],
\end{equation}
where we denote by $\Ann_{\Lambda}(M)$ the annihilator ideal of $M$ regarded as a module over the ring $\Lambda$.
Now Brumer's conjecture asserts that $\Ann_{\Z[G]} (\mu_{L}) \theta_{S}$ annihilates $\cl_{L}$.

Using $L$-values at integers $r<0$, one can define higher Stickelberger elements $\theta_{S}(r)$.
Coates and Sinnott \cite{MR0369322} conjectured that these elements can be used to construct annihilators of the higher
$K$-groups $K_{-2r}(\mathcal{O}_{S})$, where we denote by $\mathcal{O}_{S}$ the ring of $S(L)$-integers in $L$ for
any finite set $S$ of places  of $K$; here, we write $S(L)$ for the set of places of $L$ which lie above those in $S$.
However if, for example, $L$ is totally real and $r$ is even, these conjectures merely predict that zero annihilates
$K_{-2r}(\mathcal{O}_{S})$ (resp.~$\cl_{L}$) if $r<0$ (resp.~$r=0$).

In the case $r=0$, Burns \cite{MR2845620} presented a universal theory of refined Stark conjectures. In particular,
the Galois group $G$ may be non-abelian, and he uses leading terms rather than values of Artin $L$-functions to construct
conjectural nontrivial annihilators of the class group. His conjecture thereby extends the aforementioned conjecture of Brumer.
Similarly, in the case $r<0$ the second named author \cite{MR2801311} has formulated a conjecture on the annihilation of higher $K$-groups
which generalises the Coates-Sinnott conjecture.
The Quillen-Lichtenbaum conjecture relates $K$-groups to \'{e}tale cohomology, predicting that for all odd primes $p$,
integers $r<0$ and $i=0,1$ the canonical
$p$-adic Chern class maps
\[
K_{i-2r}(\mathcal{O}_{L}) \otimes_{\Z} \Z_{p} \longrightarrow
H^{2-i}_{\et}(\mathcal{O}_{L}[1/p], \Z_{p}(1-r))
\]
constructed by Soul\'{e} \cite{MR553999} are isomorphisms. Following fundamental work of Voevodsky and Rost, Weibel \cite{MR2529300} has completed the proof
of the Bloch-Kato conjecture which relates Milnor $K$-theory to \'{e}tale cohomology and implies the Quillen-Lichtenbaum conjecture.
In this way, one obtains a cohomological version of the conjecture on the annihilation of higher $K$-groups, and it is this version we will deal with later.
Both annihilation conjectures are implied by $\ETNC(L/K,r)$.

In the present article, we prove Burns' conjecture for a wide class of interesting extensions. As our method often works equally well in other situations,
we also provide new evidence for the annihilation conjecture on higher $K$-groups as well as for several other conjectures in Galois module theory.
Let us now assume that $r=0$.
We illustrate our results by again returning to the example of an extension
$L/\Q$ with $\Gal(L/\Q) \simeq \Aff(q)$ where $q$ is a power of a prime $\ell$.
As discussed above, we know $\ETNC_{p}(L/\Q,0)$ and thus
the $p$-part of Burns' conjecture for every prime $p \neq \ell$.
However, by considering certain `denominator ideals' that play a role in many annihilation conjectures, we also deduce the $\ell$-part of Burns' conjecture
(up to a factor $2$ if $\ell=2$) from the validity of the strong Stark conjecture,
even though $\ETNC_{\ell}(L/\Q,0)$ is not known in this case.
By a similar method
we prove Burns' conjecture for every Galois extension of number fields $L/K$
with $\Gal(L/K) \simeq S_{3}$.
In the case $r<0$, we prove certain cases of the aforementioned
conjecture on the annihilation of higher $K$-groups.

\subsection*{Acknowledgements}
The authors are indebted to
Frieder Ladisch for providing a proof of Lemma \ref{lem:idem-val}
and to Florian Eisele for additional comments on this topic.
The authors are grateful to James Newton for providing the reference
 \cite{MR680538}
used in Remark \ref{rmk:infinite-ETNC-family},
and to Paul Buckingham for a number of corrections and suggestions.
The second named author acknowledges financial support provided by the DFG 
within the Collaborative Research Center 701
`Spectral Structures and Topological Methods in Mathematics'.

\subsection*{Notation and conventions}
All rings are assumed to have an identity element and all modules are assumed
to be left modules unless otherwise  stated. We fix the following notation:

\medskip

\begin{tabular}{ll}
$S_{n}$ & the symmetric group on $n$ letters \\
$A_{n}$ & the alternating group on $n$ letters \\
$C_{n}$ & the cyclic group of order $n$ \\
$D_{2n}$ & the dihedral group of order $2n$\\
$Q_{8}$ & the quaternion group of order $8$\\
$V_{4}$ & the subgroup of $A_{4}$ generated by double transpositions\\
$\F_{q}$ & the finite field with $q$ elements, where $q$ is a prime power\\
$\Aff(q)$ & the affine group isomorphic to $\F_{q} \rtimes \F_{q}^{\times}$ defined in Example \ref{ex:affine}\\
$v_{p}(x)$ & the $p$-adic valuation of the rational number $x$\\
$\Irr_{F}(G)$ & the set of $F$-irreducible characters of the finite group $G$\\
$R^{\times}$ & the group of units of a ring $R$\\
$\zeta(R)$ & the centre of a ring $R$\\
$M_{m \times n} (R)$ & the set of all $m \times n$ matrices with entries in a ring $R$
\end{tabular}

\section{$p$-adic group rings and hybrid orders}

\subsection{Central idempotents in $p$-adic group rings} \label{subsec:idempotents}
Let $p$ be a prime and let $G$ be a finite group.
Let $e_{1}, \ldots, e_{t}$ be the central primitive idempotents in the group algebra $A:=\Q_{p}[G]$.
Then
\[
A = A_{1} \oplus \cdots \oplus A_{t}
\]
where $A_{i}:=Ae_{i}=e_{i}A$.
By Wedderburn's theorem each $A_{i}$ is isomorphic to an algebra of $m_{i} \times m_{i}$ matrices over
a skewfield $D_{i}$ and $F_{i} := \zeta(D_{i})$ is a finite field extension of $\Q_{p}$.
We denote the Schur index of $D_{i}$ by $s_{i}$ so that $[D_{i}:F_{i}]=s_{i}^{2}$ and set $n_{i}=m_{i}s_{i}$.

Now $e_{i} = \sum_{\chi \in \mathcal{C}_{i}} e_{\chi}$ where
each $e_{\chi}:=\chi(1)|G|^{-1}\sum_{g \in G} \chi(g^{-1})g$ is the central primitive
idempotent of $\C_{p}[G]$ corresponding to a character $\chi \in \Irr_{\C_{p}}(G)$
and each $\mathcal{C}_{i}$ is a Galois conjugacy class of such characters.
Note that $n_{i}=\chi(1)$ for any choice of $\chi \in \mathcal{C}_{i}$.
The authors are indebted to Frieder Ladisch for providing a proof of the following lemma, as well as to
Florian Eisele for additional comments.

\begin{lemma}\label{lem:idem-val}
If $e_{i} \in \Z_{p}[G]$ then $v_{p}(\chi(1)) = v_{p}(|G|)$ for some (and hence every) $\chi \in \mathcal{C}_{i}$.
\end{lemma}

\begin{proof}
Recall that an element $g \in G$ is said to be $p$-singular if its order is divisible by $p$.
Write $e_{i} = \sum_{g \in G}\epsilon_{g} g$ with $\epsilon_{g} \in \Z_{p}$ for $g \in G$.
Then \cite[Proposition 5]{MR1261587} shows that $\epsilon_{g}=0$ for every $p$-singular $g \in G$
(alternatively, one can use \cite[Proposition 3]{MR1261587} and that $e_{i}$ is central).
Let $\chi \in \mathcal{C}_{i}$ and put $H = \Gal(F_{i} / \Q_{p})$. Then
\[
e_{i} = \frac{\chi(1)}{|G|} \sum_{g \in G} \sum_{h \in H} \chi(g^{-1})^{h} g
\]
and so the character $\beta := \sum_{h \in H} \chi^{h}$ vanishes on $p$-singular elements.
Let $P$ be a Sylow $p$-subgroup of $G$.
Then $\beta$ vanishes on $P- \{1\}$, so the multiplicity of the trivial character of $P$ in the restriction $\beta_{P}$
is
\[
\langle \beta_{P} , 1_{P} \rangle = \beta(1) |P|^{-1} = \chi(1)|H||P|^{-1}.
\]
However, we also have
\[
\langle \beta_{P} , 1_{P} \rangle
 = \sum_{h \in H} \langle \chi^{h}_{P} , 1_{P} \rangle
 = |H|  \langle \chi_{P} , 1_{P} \rangle.
 \]
Therefore $\chi(1) = |P| \langle \chi_{P} , 1_{P} \rangle$, which gives the desired result.
\end{proof}

\begin{definition}\label{def:central-conductor}
Let $A$ be a finite-dimensional semisimple $\Q_{p}$-algebra and let $\Lambda \subseteq \Gamma$
be $\Z_{p}$-orders of full rank in $A.$
The \emph{central conductor} of $\Gamma$ into $\Lambda$ is defined to be
\[
\mathcal{F}(\Gamma,\Lambda):
= \{ x \in \zeta(\Gamma) \mid x\Gamma \subseteq \Lambda \}.
\]
\end{definition}

Let $\mathfrak{M}_{p}(G)$ be a maximal $\Z_{p}$-order such that $\Z_{p}[G] \subseteq \mathfrak{M}_{p}(G) \subseteq \Q_{p}[G]$.
Let $\mathcal{O}_{i}$ denote the integral closure of $\Z_{p}$ in $F_{i}$ and let
$\mathfrak{D}^{-1}(\mathcal{O}_{i}/\Z_{p})$ be the inverse different of $\mathcal{O}_{i}$ relative to $\Z_{p}$.
Then Jacobinski's central conductor formula \cite[Theorem 3]{MR0204538}
(also see \cite[Theorem 27.13]{MR632548}) says that
\begin{equation}\label{eq:conductor-formula}
\mathcal{F}_{p}(G) := \mathcal{F}(\mathfrak{M}_{p}(G), \Z_{p}[G]) = \bigoplus_{i=1}^{t} |G|n_{i}^{-1} \mathfrak{D}^{-1}(\mathcal{O}_{i}/\Z_{p}).
\end{equation}
This is independent of the particular choice of maximal order $\mathfrak{M}_{p}(G)$.
The idea to use \eqref{eq:conductor-formula} in the proof of the following proposition is adapted from \cite[\S 56, Exercise 10]{MR892316}.

\begin{prop}\label{prop:equiv-idem}
Fix $i$ with $1 \leq i \leq t$. Then the following are equivalent:
\begin{enumerate}
\item $e_{i} \in \Z_{p}[G]$,
\item $e_{i} \in \Z_{p}[G]$ and $\Z_{p}[G]e_{i}$ is a maximal $\Z_{p}$-order,
\item $e_{i} \in \Z_{p}[G]$ and $F_{i}/\Q_{p}$ is unramified,
\item $v_{p}(\chi(1)) = v_{p}(|G|)$ for some (and hence every) $\chi \in \mathcal{C}_{i}$.
\end{enumerate}
Furthermore, if these equivalent conditions hold then $s_{i}=1$.
\end{prop}

\begin{proof}
It is clear that (i) is implied by each of the other conditions.
That (i) implies (iv) is  Lemma \ref{lem:idem-val}.
Suppose that (iv) holds.
Then from \eqref{eq:conductor-formula} we have
\begin{equation}\label{eq:trunc-cond}
\mathcal{F}(\mathfrak{M}_{p}(G)e_{i}, \Z_{p}[G]e_{i})
= |G|n_{i}^{-1} \mathfrak{D}^{-1}(\mathcal{O}_{i}/\Z_{p}).
\end{equation}
Since $v_{p}(\chi(1)) = v_{p}(|G|)$ for some  $\chi \in \mathcal{C}_{i}$ we have
that $|G|^{-1}n_{i} = |G|^{-1}\chi(1) \in \Z_{p}^{\times}$.
Furthermore, $1 \in \mathfrak{D}^{-1}(\mathcal{O}_{i}/\Z_{p})$ and so we must have
that $1$ is in the ideal in \eqref{eq:trunc-cond}, which forces it to be the trivial ideal.
This in turn forces $\mathfrak{D}^{-1}(\mathcal{O}_{i}/\Z_{p})$ to be trivial.
Hence  $\Z_{p}[G]e_{i}$ is a maximal $\Z_{p}$-order and $F_{i}/\Q_{p}$ is unramified, so (ii) and (iii) hold.
The last claim is \cite[Theorem 5]{MR0302752} (see Remark \ref{rmk:language-mod-rep-thy} for explanation of terminology).
\end{proof}

\begin{remark}\label{rmk:language-mod-rep-thy}
In the language of modular representation theory, when $v_{p}(\chi(1))=v_{p}(|G|)$ we say that ``$\chi$ belongs to a $p$-block of defect zero''.
When working over an extension of $\Q_{p}$ that is ``sufficiently large relative to $G$'', the theory of such blocks is well-understood;
see \cite[\S 56]{MR892316} or \cite[\S 16.4]{MR0450380}, for example. 
Proposition \ref{prop:equiv-idem} can also be deduced from \cite[Proposition 46]{MR0450380} using Lemma \ref{lem:idem-val} and \cite[Theorem]{MR551444}.
\end{remark}

\subsection{Hybrid $p$-adic group rings}
For a normal subgroup $N$ of $G$, let $e_{N} = |N|^{-1}\sum_{\sigma \in N} \sigma$
be the associated central trace idempotent in the group algebra $\Q_{p}[G]$.
Then there is a ring isomorphism $\Z_{p}[G]e_{N} \simeq \Z_{p}[G/N]$.
In particular, if $G'$ is the commutator subgroup of $G$ and
$G^{\mathrm{ab}}=G/G'$ is the maximal abelian quotient,
then $\Z_{p}[G]e_{G'} \simeq \Z_{p}[G^{\mathrm{ab}}]$.

\begin{definition}\label{def:max-comm-hybrid}
Let $\mathfrak{M}_{p}(G)$ be a maximal $\Z_{p}$-order
such that $\Z_{p}[G] \subseteq \mathfrak{M}_{p}(G) \subseteq \Q_{p}[G]$
and let $N$ be a normal subgroup of $G$.
Define the $N$-\emph{hybrid order} of $\Z_{p}[G]$ and $\mathfrak{M}_{p}(G)$ to be
$\mathfrak{M}_{p}(G,N)=\Z_{p}[G] e_{N} \oplus \mathfrak{M}_{p}(G)(1-e_{N})$.
We say that $\Z_{p}[G]$ is $N$-\emph{hybrid} if $\Z_{p}[G] =  \mathfrak{M}_{p}(G,N)$
for some choice of $\mathfrak{M}_{p}(G)$. Let $J_{p}(G,N) = \{ i \mid e_{i}e_{N}=0 \}$.
\end{definition}

\begin{remark}
The group ring $\Z_{p}[G]$ is itself maximal if and only if $p$ does not divide $|G|$
if and only if it is $G$-hybrid.
\end{remark}

\begin{prop}\label{prop:hybrid-criterion}
The group ring $\Z_{p}[G]$ is $N$-hybrid if and only if
the equivalent conditions of Proposition \ref{prop:equiv-idem} hold for each $i \in J_{p}(G,N)$.
In particular,  $\Z_{p}[G]$ is $N$-hybrid if and only if
for every $\chi \in \Irr_{\C_{p}}(G)$ such that $N \not \leq \ker \chi$
we have $v_{p}(\chi(1))=v_{p}(|G|)$.
\end{prop}

\begin{proof}
This is clear once one notes that $\mathfrak{M}_{p}(G)(1-e_{N}) = \oplus_{i \in J_{p}(G,N)} \mathfrak{M}_{p}(G)e_{i}$.
\end{proof}

%

\begin{prop}\label{prop:hybrid-properties}
Suppose $\Z_{p}[G]$ is $N$-hybrid. Then
\begin{enumerate}
\item $p$ does not divide $|N|$,
\item for each $i \in J_{p}(G,N)$ the extension $F_{i}/\Q_{p}$ is unramified and $s_{i}=1$,
\item there is a ring isomorphism
\[
\Z_{p}[G] \simeq \Z_{p}[G/N] \oplus \bigoplus_{i \in J_{p}(G,N)} M_{m_{i} \times m_{i}}(\mathcal{O}_{i}),
\]
\item for any normal subgroup $K \unlhd G$ with $K \leq N$, $\Z_{p}[G]$ is also $K$-hybrid.
\end{enumerate}
\end{prop}

\begin{proof}
For (i), note that we have $e_{N} = |N|^{-1}\sum_{\sigma \in N} \sigma \in \Z_{p}[G]$ and so $p$ does not divide $|N|$.
For (ii), use Proposition \ref{prop:hybrid-criterion} with Proposition \ref{prop:equiv-idem}.
Part (iii) follows from part (ii) and \cite[Theorem (17.3)]{MR1972204}.
Part (iv) follows from Proposition \ref{prop:hybrid-criterion} and the observation that $J_{p}(G,K) \subseteq J_{p}(G,N)$.
\end{proof}

\begin{lemma}\label{lem:add-a-group}
Let $H$ be a finite group of order prime to $p$.
If $\Z_{p}[G]$ is $N$-hybrid then $\Z_{p}[G \times H]$ is $(N \times \{1 \})$-hybrid.
\end{lemma}

\begin{proof}
Each $\chi \in \Irr_{\C_{p}}(G \times H)$ such that
$N \times 1 \not \leq \ker \chi$ is the product of characters
$\psi \in \Irr_{\C_{p}}(G)$ and $\zeta \in \Irr_{\C_{p}}(H)$ with $N \not \leq \ker \psi$.
Hence the desired result follows from Proposition \ref{prop:hybrid-criterion} and the equality
$v_{p}(\chi(1)) = v_{p}(\psi(1)) = v_{p}(|G|) = v_{p}(|G \times H|)$.
\end{proof}

\subsection{Frobenius groups}\label{subsec:frobenius-groups}
We recall the definition and some basic facts about Frobenius groups and then use them to
provide many examples of hybrid $p$-adic group rings.

\begin{definition}
A \emph{Frobenius group} is a finite group $G$ with a proper nontrivial subgroup $H$
such that $H \cap gHg^{-1}=\{ 1 \}$ for all $g \in G-H$,
in which case $H$ is called a \emph{Frobenius complement}.
\end{definition}

\begin{theorem}\label{thm:frob-kernel}
A Frobenius group $G$ contains a unique normal subgroup $N$, known as the Frobenius kernel, such that
$G$ is a semidirect product $N \rtimes H$. Furthermore:
\begin{enumerate}
\item $|N|$ and $[G:N]=|H|$ are relatively prime.
\item The Frobenius kernel $N$ is nilpotent.
\item If $K \unlhd G $ then either $K \unlhd N$ or $N \unlhd K$.
\item If $\chi \in \Irr_{\C}(G)$ such that  $N \not \leq \ker \chi$ then $\chi= \mathrm{Ind}_{N}^{G}(\psi)$ for some $1 \neq \psi \in \Irr_{\C}(N)$.
\end{enumerate}
\end{theorem}

\begin{proof}
For (i) and (iv) see \cite[\S 14A]{MR632548}.
For (ii) see \cite[10.5.6]{MR1357169} and for (iii) see  \cite[Exercise 7, \S 8.5]{MR1357169}.
\end{proof}

\begin{theorem}\label{thm:frob-equiv}
The following statements are equivalent:
\begin{enumerate}
\item $G$ is a Frobenius group.
\item $G$ contains a proper nontrivial normal subgroup $N$ such that for each $n \in N$, $n \neq 1$, the centraliser of $n$ in $G$ is contained in $N$.
\item $G$ can be expressed as a nontrivial semidirect product $N \rtimes H$ such that the action of $H$ on $N$
is fixed-point-free (i.e. $n^{h} \neq n$ whenever $h,n \neq 1$, $h \in H$, $n \in N$).
\end{enumerate}
\end{theorem}

\begin{proof}
See \cite[\S 4.6]{MR2599132}, for example.
\end{proof}

\begin{prop}\label{prop:frob-N-hybrid}
Let $G$ be a Frobenius group with Frobenius kernel $N$.
Then for every prime $p$ not dividing $|N|$, the group ring $\Z_{p}[G]$ is $N$-hybrid.
\end{prop}

\begin{proof}
Let $\chi \in \Irr_{\C_{p}}(G)$ such that $N \not \leq \ker \chi$.
Then by Theorem \ref{thm:frob-kernel}(iv) $\chi$ is induced from a nontrivial irreducible character
of $N$ and so $\chi(1)$ is divisible by $[G:N]$.
However, $|N|$ and $[G:N]$ are relatively prime by Theorem \ref{thm:frob-kernel}(i) and so
the desired result now follows from Proposition \ref{prop:hybrid-criterion}.
\end{proof}

\begin{example}\label{ex:inversion}
Let $A$ be a nontrivial finite abelian group of odd order and let $C_{2}$ act on $A$ by inversion.
Then the semidirect product $G = A \rtimes C_{2}$ is a Frobenius group and so $\Z_{2}[G]$ is $A$-hybrid
and thus is isomorphic to $\Z_{2}[C_{2}] \oplus \mathfrak{M}_{2}(G,A)(1-e_{A})$.
In particular, if $n$ is odd then one can take $G=D_{2n}$ and $A=G' \simeq C_{n}$, the subgroup of rotations.
\end{example}

\begin{example}\label{ex:metacyclic}
Let $p<q$ be distinct primes and assume that $p \mid (q-1)$.
Then there is an embedding $C_{p} \hookrightarrow \Aut(C_{q})$ and so there is
a fixed-point-free action of $C_{p}$ on $C_{q}$.
Hence the corresponding semidirect product $G = C_{q} \rtimes C_{p}$ is a Frobenius group
by Theorem \ref{thm:frob-equiv}(iii), and so $\Z_{p}[G]$ is $G'$-hybrid with $G' = C_{q}$.
\end{example}

\begin{example}\label{ex:affine}
Let $q$ be a prime power and let $\F_{q}$ be the finite field with $q$ elements.
The group $\Aff(q)$ of affine transformations on $\F_{q}$ is the group of transformations
of the form $x \mapsto ax +b$ with $a \in \F_{q}^{\times}$ and $b \in \F_{q}$.
Let $G=\Aff(q)$ and let $N=\{ x \mapsto x+b \mid b \in \F_{q} \}$.
Then $G$ is a Frobenius group with Frobenius kernel $N=G' \simeq \F_{q}$ and is isomorphic to
the semidirect product $\F_{q} \rtimes \F_{q}^{\times}$ with the natural action.
Furthermore, $G/N \simeq \F_{q}^{\times} \simeq C_{q-1}$ and $G$ has precisely one
non-linear irreducible complex character, which is rational-valued and of degree $q-1$.
Hence for every prime $p$ not dividing $q$, we have that $\Z_{p}[G]$ is $N$-hybrid
and is isomorphic to $\Z_{p}[C_{q-1}] \oplus M_{(q-1) \times (q-1)}(\Z_{p})$.
Note that in particular $\Aff(3) \simeq S_{3}$ and $\Aff(4) \simeq A_{4}$.
Thus $\Z_{2}[S_{3}] \simeq \Z_{2}[C_{2}] \oplus M_{2 \times 2}(\Z_{2})$
and $\Z_{3}[A_{4}] \simeq \Z_{3}[C_{3}] \oplus M_{3 \times 3}(\Z_{3})$.
\end{example}

\begin{example}\label{ex:frob72}
Let $p=2$ and let $G = N \rtimes Q_{8}$ where $N$ is the $2$-dimensional irreducible representation of $Q_{8}$ over $\F_{3}$
(so $N \simeq C_{3} \times C_{3}$). Thus $G$ is a Frobenius group with Frobenius kernel $N$ and so
$\Z_{2}[G]$ is $N$-hybrid.
The unique non-linear complex irreducible character of $G$ not inflated from $Q_{8}$ is rational-valued and of degree $8$. Hence we have $\Z_{2}[G] \simeq \Z_{2}[Q_{8}] \oplus M_{8 \times 8}(\Z_{2})$.
Further examples of groups of this type are given in \cite{MR2104937}.
\end{example}

\begin{example}\label{ex:S4-V4}
Let $p=3$, $G=S_{4}$ and $N=V_{4}$. Then $G/N \simeq S_{3}$ and the only two complex irreducible characters of $G$ not inflated from characters of $S_{3}$ are of degree $3$ and are rational-valued.
Hence $\Z_{3}[S_{4}]$ is $V_{4}$-hybrid and is isomorphic to
$\Z_{3}[S_{3}] \oplus M_{3 \times 3}(\Z_{3}) \oplus  M_{3 \times 3}(\Z_{3})$.
However, the only proper nontrivial normal subgroups of $S_{4}$ are $A_{4}$ and $V_{4}$, and so
by Theorem \ref{thm:frob-kernel}(i) we see that $S_{4}$ is \emph{not} a Frobenius group.
\end{example}

\begin{remark}
If $\Z_{p}[G]$ is $N$-hybrid then $p \nmid |N|$ by Proposition \ref{prop:hybrid-properties}(i).
If $p \nmid |G'|$ then $\Z_{p}[G]$ is a direct sum of matrix rings over commutative $\Z_{p}$-algebras
by \cite[Corollary]{MR704622}. However, in this case $\Z_{p}[G]$ is not necessarily $G'$-hybrid.
For example, if $G=C_{3} \times A_{4}$ then $G'=\{ 1 \} \times V_{4}$, which has order prime to $3$,
but by Example \ref{ex:affine} we have
$\Z_{3}[C_{3} \times A_{4}] \simeq \Z_{3}[C_{3}] \otimes_{\Z_{3}} \Z_{3}[A_{4}]
\simeq \Z_{3}[C_{3} \times C_{3}] \oplus M_{3 \times 3}(\Z_{3}[C_{3}])$ where $\Z_{3}[C_{3}]$
and hence $M_{3 \times 3}(\Z_{3}[C_{3}])$ is not maximal.
\end{remark}

\section{Relative algebraic $K$-groups and weakly hybrid orders}\label{sec:rel-k-group-weak-hybrid}

For further details and background on algebraic $K$-theory used in this section, we refer the reader to
\cite{MR892316}, \cite{MR0245634}, \cite[\S 2]{breuning-thesis} or \cite{MR2564571}.

\subsection{Algebraic $K$-theory}\label{subsec:K-theory}
Let $R$ be a noetherian integral domain of characteristic $0$ with field of fractions $E$.
Let $A$ be a finite-dimensional semisimple $E$-algebra and let $\mathfrak{A}$ be an $R$-order in $A$.
Let $\PMod(\mathfrak{A})$ denote the category of finitely generated projective left $\mathfrak{A}$-modules.
We write $K_{0}(\mathfrak{A})$ for the Grothendieck group of $\PMod(\mathfrak{A})$
(see \cite[\S 38]{MR892316})
and $K_{1}(\mathfrak{A})$ for the Whitehead group (see \cite[\S 40]{MR892316}).
For any field extension $F$ of $E$ we set $A_{F} := F \otimes_{E} A$.
Let $K_{0}(\mathfrak{A}, F)$ denote the relative
algebraic $K$-group associated to the ring homomorphism $\mathfrak{A} \hookrightarrow A_{F}$.
We recall that $K_{0}(\mathfrak{A}, F)$ is an abelian group with generators $[X,g,Y]$ where
$X$ and $Y$ are finitely generated projective $\mathfrak{A}$-modules
and $g:F \otimes_{R} X \rightarrow F \otimes_{R} Y$ is an isomorphism of $A_{F}$-modules;
for a full description in terms of generators and relations, we refer the reader to \cite[p.\ 215]{MR0245634}.
Furthermore, there is a long exact sequence of relative $K$-theory
\begin{equation}\label{eq:long-exact-seq}
K_{1}(\mathfrak{A}) \longrightarrow K_{1}(A_{F}) \longrightarrow K_{0}(\mathfrak{A},F)
\longrightarrow K_{0}(\mathfrak{A}) \longrightarrow K_{0}(A_{F})
\end{equation}
(see  \cite[Chapter 15]{MR0245634}).
We define $DT(\mathfrak{A})$ to be the torsion subgroup of $K_{0}(\mathfrak{A},E)$.

\subsection{Functorialities}\label{subsec:K-func}
We follow \cite[\S 3.5]{MR1884523} in describing the functorial behaviour of relative algebraic $K$-groups.
Let $\mathfrak{B}$ be an $R$-order in a finite-dimensional semisimple $E$-algebra $B$
and let $\rho: \mathfrak{A} \rightarrow \mathfrak{B}$ be a ring homomorphism.
The scalar extension functor $\mathfrak{B} \otimes_{\mathfrak{A}} -$ induces a natural homomorphism
\[\
\rho_{*} : K_{0}(\mathfrak{A},F) \longrightarrow K_{0}(\mathfrak{B},F)
\]
which sends $[X,g,Y]$ to $[\mathfrak{B} \otimes_{\mathfrak{A}} X , 1 \otimes g, \mathfrak{B} \otimes_{\mathfrak{A}} Y]$.
If $\mathfrak{B}$ is a projective $\mathfrak{A}$-module via $\rho$, then there also exists a homomorphism in the reverse
direction
\[
\rho^{*} : K_{0}(\mathfrak{B},F) \longrightarrow K_{0}(\mathfrak{A},F)
\]
which is simply induced by restriction of scalars.

If $\mathfrak{A}$ is commutative and $\mathfrak{B} =M_{n}(\mathfrak{A})$ is a matrix algebra over $\mathfrak{A}$,
then we define $e \in \mathfrak{B}$ to be the matrix with the entry in the top left hand corner equal to $1$
and all other entries equal to zero. In this case the exact functor $X \mapsto eX$ induces an equivalence
of exact categories $\mu:\PMod(\mathfrak{B}) \rightarrow \PMod(\mathfrak{A})$ and hence also an isomorphism
\begin{equation}\label{eq:morita-isom}
\mu_{*} : K_{0}(\mathfrak{B},F) \stackrel{\sim}{\longrightarrow} K_{0}(\mathfrak{A},F).
\end{equation}

We now consider several special cases of particular interest. Let $G$ be a finite group with subgroup $H$ and normal subgroup $N$.
\begin{itemize}
\item Restriction: $\res^{G}_{H} := \rho^{*}$ where $\rho:R[H] \hookrightarrow R[G]$ is inclusion.
\item Induction:   $\ind^{G}_{H} := \rho_{*}$ where $\rho:R[H] \hookrightarrow R[G]$ is inclusion.
\item Quotient: $\quot^{G}_{G/N} := \rho_{*}$ where $\rho:R[G] \rightarrow R[G/N]$ is the natural projection.
\item Projection: $\proj^{G}_{G/N} := \rho_{*}$ where
$\mathfrak{A}$ is an $R$-order containing $R[G]$
such that $e_{N}\mathfrak{A}=e_{N}R[G] \simeq R[G/N]$
and $\rho:\mathfrak{A} \rightarrow R[G/N]$ is the natural projection.
\end{itemize}
In the case that $\mathfrak{A}=R[G]$ and $e_{N} \in R[G]$ the maps $\quot^{G}_{G/N}$
and $\proj^{G}_{G/N}$ coincide.

\subsection{Descriptions of torsion subgroups}
Let $p$ be a prime and let $A$ be a finite-dimensional semisimple $\Q_{p}$-algebra.
Let $\mathfrak{A}_{p}$ be a $\Z_{p}$-order contained in $A$.
Let $\nr:A^{\times} \longrightarrow \zeta(A)^{\times}$ denote the reduced norm map (see \cite[\S 7D]{MR632548}).


\begin{prop}\label{prop:torsion-units}
Let
$\mathfrak{M}_{p}$ be any maximal $\Z_{p}$-order such that $\mathfrak{A}_{p} \subseteq \mathfrak{M}_{p} \subseteq A$.
Then
\[
DT(\mathfrak{A}_{p}) \simeq \frac{\zeta(\mathfrak{M}_{p})^{\times}}{\nr(\mathfrak{A}_{p}^{\times})},
\]
and this group is finite. In particular, $DT(\mathfrak{M}_{p})$ is trivial.
\end{prop}

\begin{proof}
This is a special case of \cite[Theorem 2.4]{MR2564571}.
\end{proof}

\begin{prop}\label{prop:torsion-kernel}
Let $G$ be a finite group and let $\mathfrak{M}_{p}(G)$ be any maximal $\Z_{p}$-order
such that $\Z_{p}[G] \subseteq \mathfrak{M}_{p}(G) \subseteq \Q_{p}[G]$.
Then
\[
DT(\Z_{p}[G]) = \ker(K_{0}(\Z_{p}[G], \Q_{p}) \stackrel{\theta}{\longrightarrow} K_{0}( \mathfrak{M}_{p}(G), \Q_{p}) )
\]
where $\theta = [\mathfrak{M}_{p}(G) \otimes_{\Z_{p}[G]} - ]$ is the map induced by extension of scalars.
\end{prop}

\begin{proof}
By \cite[Theorem 2.3(i)]{MR2564571} this is a special case of
\cite[Theorem 2.4(iii)]{MR2564571}.
\end{proof}

\subsection{Maps between torsion subgroups in the group ring case}
Let $p$ be a prime and let $G$ be a finite group. 
By a well-known theorem of Swan (see \cite[Theorem (32.1)]{MR632548}) the 
map $K_{0}(\Z_{p}[G]) \longrightarrow K_{0}(\Q_{p}[G])$ induced by extension of scalars is injective. Thus from \eqref{eq:long-exact-seq} we obtain an exact sequence
\begin{equation}\label{eq:group-ring-K-exact-seq}
K_{1}(\Z_{p}[G]) \longrightarrow K_{1}(\Q_{p}[G]) \longrightarrow K_{0}(\Z_{p}[G],\Q_{p})
\longrightarrow 0,
\end{equation}
which is functorial with respect to restriction and quotient maps.
The reduced norm map induces an isomorphism $K_{1}(\Q_{p}[G]) \longrightarrow \zeta(\Q_{p}[G])^{\times}$
(use \cite[Theorem (45.3)]{MR892316}) and $\nr(K_{1}(\Z_{p}[G]))=\nr((\Z_{p}[G])^{\times})$ 
(this follows from  \cite[Theorem (40.31)]{MR892316}). 
Hence from \eqref{eq:group-ring-K-exact-seq} we obtain an exact sequence
\begin{equation}\label{eq:group-ring-units-seq}
(\Z_{p}[G])^{\times} \stackrel{\nr}{\longrightarrow} \zeta(\Q_{p}[G])^{\times} \longrightarrow K_{0}(\Z_{p}[G],\Q_{p})
\longrightarrow 0.
\end{equation}
Now let $\mathfrak{M}_{p}(G)$ be a maximal $\Z_{p}$-order such that $\Z_{p}[G] \subseteq \mathfrak{M}_{p}(G) \subseteq \Q_{p}[G]$.
Then by restricting the middle map of \eqref{eq:group-ring-units-seq} we obtain an exact sequence
\begin{equation}\label{eq:group-ring-DT-seq}
(\Z_{p}[G])^{\times} \stackrel{\nr}{\longrightarrow} \zeta(\mathfrak{M}_{p}(G))^{\times} \longrightarrow DT(\Z_{p}[G])
\longrightarrow 0,
\end{equation}
which is again functorial with respect to restriction and quotient maps; 
moreover, this sequence gives a proof of Proposition \ref{prop:torsion-units} in the case $\mathfrak{A}_{p}=\Z_{p}[G]$.

\begin{prop}\label{prop:quotient-surjective}
Let $p$ be a prime and let $G$ be a finite group with normal subgroup $N$. Then the quotient map
$\quot^{G}_{G/N}: DT(\Z_{p}[G]) \longrightarrow DT(\Z_{p}[G/N])$
is surjective.
\end{prop}

\begin{proof}
Let $\mathfrak{M}_{p}(G)$ be a maximal $\Z_{p}$-order such that $\Z_{p}[G] \subseteq \mathfrak{M}_{p}(G) \subseteq \Q_{p}[G]$.
Then $\mathfrak{M}_{p}(G)$ decomposes into
\[
\mathfrak{M}_{p}(G) = \mathfrak{M}_{p}(G) e_{N} \oplus \mathfrak{M}_{p}(G) (1-e_{N}),
\]
where $\mathfrak{M}_{p}(G/N) := \mathfrak{M}_{p}(G) e_{N}$ is a maximal order in $\Q_{p}[G] e_{N} \simeq \Q_{p}[G/N]$.
Hence we obtain a surjection $\zeta(\mathfrak{M}_{p}(G))\mal \onto \zeta(\mathfrak{M}_{p}(G/N))\mal$.
Since \eqref{eq:group-ring-DT-seq} is functorial with respect to quotient maps, 
we have a commutative diagram
\[
\xymatrix{
{\zeta(\mathfrak{M}_{p}(G))^{\times}} \ar@{->>}[r] \ar@{->>}[d] & DT(\Z_{p}[G]) \ar@{->}[d]^{\quot^{G}_{G/N}}\\
{\zeta(\mathfrak{M}_{p}(G/N))^{\times}} \ar@{->>}[r] & DT(\Z_{p}[G/N]).
}
\]
Hence we see that the right vertical arrow must also be surjective.
\end{proof}

\begin{prop}\label{prop:restriction-surjective}
Let $p$ be a prime and let $G$ be a finite group with a subgroup $H \simeq C_{p}$.
Then the restriction map
$
\res^{G}_{H}: DT(\Z_{p}[G]) \longrightarrow DT(\Z_{p}[H])
$
is surjective.
\end{prop}

\begin{proof}
We henceforth identify $H$ and $C_{p}$.
Let $\mathfrak{M}_{p}(C_{p})$ and $\mathfrak{M}_{p}(G)$ be maximal $\Z_{p}$-orders 
such that $\Z_{p}[C_{p}] \subseteq \mathfrak{M}_{p}(C_{p}) \subseteq \Q_{p}[C_{p}]$ and
$\Z_{p}[G] \subseteq \mathfrak{M}_{p}(G) \subseteq \Q_{p}[G]$.
By \cite[Corollary 8.2]{MR2564571} and its proof we have a commutative diagram
\[
\xymatrix{
{\mathfrak{M}_{p}(C_{p})^{\times}} \ar@{->>}[r] \ar@{->}[d]^{\simeq} & DT(\Z_{p}[C_{p}]) \ar@{->}[d]^{\simeq}\\
{K_{1}(\Z_{p}) \times K_{1}(\Z_{p}[\zeta_{p}])} \ar@{->>}[r]^{\quad \quad \rho_{\ast}} & K_{1}(\F_{p}),
}
\]
where $\rho_{\ast}$ is induced from $\rho_{1}: \Z_{p} \rightarrow \Z_{p} / p \Z_{p} \simeq \F_{p}$
and  $\rho_{2}: \Z_{p}[\zeta_{p}] \rightarrow \Z_{p}[\zeta_{p}] / (1-\zeta_{p}) \simeq \F_{p}$ by functoriality.
More precisely, we have $\rho_{\ast}(x,y) = K_{1}(\rho_{1})(x) K_{1}(\rho_{2})(y)^{-1}$
for all $x \in K_{1}(\Z_{p})$, $y \in K_{1}(\Z_{p}[\zeta_{p}])$. Observe that
$K_{1}(\rho_{1}): K_{1}(\Z_{p}) \rightarrow K_{1}(\F_{p})$ is still surjective.
By the functoriality of \eqref{eq:group-ring-DT-seq} with respect to restriction, 
we have a commutative diagram
\[
\xymatrix{
{\zeta(\mathfrak{M}_{p}(G))^{\times}} \ar@{->>}[r] \ar@{->}[d]^{\res^{G}_{C_{p}}} \ar@{->}[dr]^{\psi} & DT(\Z_{p}[G]) \ar@{->}[d]^{\res^{G}_{C_{p}}}\\
{\mathfrak{M}_{p}(C_{p})^{\times}} \ar@{->>}[r] & DT(\Z_{p}[C_{p}]).
}
\]
It suffices to show that $\psi$ is surjective. For this let $x \in DT(\Z_{p}[C_{p}]) \simeq K_{1}(\F_{p})$ be arbitrary.
Let $z \in \Z_{p}^{\times} \simeq K_{1}(\Z_{p})$ be a preimage under $K_{1}(\rho_{1}): K_{1}(\Z_{p}) \twoheadrightarrow K_{1}(\F_{p})$.
Then $z e_{G} + (1 - e_{G})$ lies in $\zeta(\mathfrak{M}_{p}(G))^{\times}$. Moreover, the explicit description of the restriction map
given in \cite[\S 6.1]{MR2564571} shows that
\[
    \res^{G}_{C_{p}}(z e_{G} + (1 - e_{G})) = z e_{C_{p}} + (1 - e_{C_{p}}).
\]
Therefore $\psi(z e_{G} + (1 - e_{G})) = x$ as desired.
\end{proof}

\begin{corollary}\label{cor:trivial-DT}
Let $p$ be a prime and let $G$ be a finite group.
If $DT(\Z_{p}[G])$ is a $p$-group (in particular, if $DT(\Z_{p}[G])=0$), 
then either $\Z_{p}[G]$ is maximal or $p=2$.
\end{corollary}

\begin{proof}
Assume that $\Z_{p}[G]$ is not maximal, or equivalently, that $p$ divides $|G|$.
Then $G$ has a subgroup of order $p$.
By \cite[Corollary 8.2]{MR2564571} we have $DT(\Z_{p}[C_{p}]) \simeq \F_{p}^{\times} \simeq C_{p-1}$,
and so Proposition \ref{prop:restriction-surjective} gives the desired result.
\end{proof}

\subsection{Weakly hybrid orders}\label{subsec:wh}
We now introduce the notion of a weakly hybrid order.

\begin{definition}\label{def:weakly-hybrid}
Let $p$ be a prime and let $G$ be a finite group with normal subgroup $N$.
Let $\mathfrak{A}_{p}$ be a $\Z_{p}$-order such that $\Z_{p}[G] \subseteq \mathfrak{A}_{p} \subseteq \Q_{p}[G]$.
The order $\mathfrak{A}_{p}$ is said to be \emph{weakly} $N$-\emph{hybrid} if
(i) $e_{N} \in \mathfrak{A}_{p}$, (ii) $e_{N}\mathfrak{A}_{p}=e_{N}\Z_{p}[G] \simeq \Z_{p}[G/N]$ and
(iii) $DT(\mathfrak{A}_{p}(1-e_{N}))$ is trivial.
\end{definition}

The following lemma shows that Definition \ref{def:weakly-hybrid} is a generalisation of
Definition \ref{def:max-comm-hybrid}.

\begin{lemma}
If $\mathfrak{A}_{p}$ is $N$-hybrid then it is weakly $N$-hybrid.
\end{lemma}

\begin{proof}
By  Definition \ref{def:max-comm-hybrid} we have $e_{N} \in \mathfrak{A}_{p}$ and $\mathfrak{A}_{p}(1-e_{N}) = \mathfrak{M}_{p}(1-e_{N})$
for some choice of maximal $\Z_{p}$-order $\mathfrak{M}_{p}$ such that  $\Z_{p}[G] \subseteq \mathfrak{A}_{p} \subseteq \mathfrak{M}_{p} \subseteq \Q_{p}[G]$.
So by Proposition \ref{prop:torsion-units} we see that $DT(\mathfrak{A}_{p}(1-e_{N}))$ vanishes.
\end{proof}

The following proposition is the key reason for the interest in weakly hybrid orders.

\begin{prop}\label{prop:wh-quot}
If $\mathfrak{A}_{p}$ is weakly $N$-hybrid then the map
\[
\proj_{G/N}^{G}: DT(\mathfrak{A}_{p}) \longrightarrow DT(\Z_{p}[G/N])
\]
is an isomorphism.  If in addition $\mathfrak{A}_{p}=\Z_{p}[G]$ then $\proj_{G/N}^{G}=\quot_{G/N}^{G}$.
\end{prop}

\begin{proof}
By Definition \ref{def:weakly-hybrid} we have $\mathfrak{A}_{p}=\Z_{p}[G]e_{N}\ \oplus \mathfrak{A}_{p}(1-e_{N})$, which induces a natural decomposition
\[
DT(\mathfrak{A}_{p})  = DT(\Z_{p}[G]e_{N}) \oplus DT(\mathfrak{A}_{p}(1-e_{N})).
\]
The map $\proj_{G/N}^{G}$ is projection onto the first summand and
the second summand is trivial by hypothesis.
The last claim is the observation made at the end of \S \ref{subsec:K-func}.
\end{proof}

\begin{lemma}\label{lem:weak-hybrid-p-does-not-divide-N}
If $\Z_{p}[G]$ is weakly $N$-hybrid then $p$ does not divide $|N|$.
\end{lemma}

\begin{proof}
This is the same argument as that given in the proof of Proposition \ref{prop:hybrid-properties}(i).
\end{proof}

The following lemma is a generalisation of \cite[Lemma 8.2]{MR2005875}.

\begin{lemma}\label{lem:trivial-DT2}
Let $A$ be a finite abelian group of odd order and let $C_{2}$ act on $A$ by inversion.
Let $G$ be the semidirect product $A \rtimes C_{2}$.
In particular, one may take $G=C_{2}$ or $D_{2n}$ for $n$ odd.
Then $DT(\Z_{2}[G])$ is trivial.
\end{lemma}

\begin{proof}
It is well-known that $DT(\Z_{2}[C_{2}])$ is trivial
(this is a special case of \cite[Corollary 8.2]{MR2564571}, for example).
By Example \ref{ex:inversion}, $\Z_{2}[G]$ is $A$-hybrid when $A$ is nontrivial and so by
Proposition \ref{prop:wh-quot} we have  $DT(\Z_{2}[G]) \simeq DT(\Z_{2}[C_{2}])=0$.
\end{proof}

\begin{lemma}\label{lem:weak-hybrid-products}
Suppose that $\Z_{2}[G]$ is $N$-hybrid and that
for every $\chi \in \Irr_{\C_{2}}(G)$ such that $N \not \leq \ker \chi$
we have that $\Q_{2}(\chi)=\Q_{2}$. Let $H$ be a finite group such that $DT(\Z_{2}[H])$ is trivial.
Then $\Z_{2}[G \times H]$ is weakly $N \times \{ 1 \}$-hybrid.
If in addition $H$ is  of even order then $\Z_{2}[G \times H]$ is not $N \times \{ 1 \}$-hybrid.
\end{lemma}

\begin{remark}
Corollary \ref{cor:trivial-DT} shows that we need only consider the case $p=2$ in Lemma
\ref{lem:weak-hybrid-products}; in the case $p>2$ the analogous result reduces to Lemma \ref{lem:add-a-group}.
\end{remark}

\begin{proof}[Proof of Lemma \ref{lem:weak-hybrid-products}]
By Proposition \ref{prop:hybrid-properties} and the hypotheses we have that $\Z_{2}[G](1-e_{N})$ is a direct sum of
orders of the form $M_{n \times n}(\Z_{2})$.
Thus $\Z_{2}[G \times H](1-e_{N \times \{1 \}}) \simeq \Z_{2}[G](1-e_{N}) \otimes_{\Z_{2}} \Z_{2}[H]$
is a direct sum of orders of the form $M_{n \times n}(\Z_{2}) \otimes_{\Z_{2}} \Z_{2}[H] \simeq M_{n \times n}(\Z_{2}[H])$.
However, by \eqref{eq:morita-isom} we have
$DT(M_{n \times n}(\Z_{2}[H])) = DT(\Z_{2}[H])=0$. As $DT(-)$ respects direct sums, this gives the first claim.
If $H$ is of even order then $\Z_{2}[H]$ is not maximal and so the summands of the form $M_{n \times n}(\Z_{2}[H])$
are not maximal either.
\end{proof}

\begin{example} \label{ex:D12-weakly-hybrid}
By combining Lemmas \ref{lem:trivial-DT2} and \ref{lem:weak-hybrid-products} one can give many examples
of group rings $\Z_{2}[G]$ that are weakly $N$-hybrid but not $N$-hybrid. We give just one example here.
In Example \ref{ex:affine}, it was shown that $\Z_{2}[S_{3}]$ is $A_{3}$-hybrid and
$\Z_{2}[S_{3}] \simeq \Z_{2}[C_{2}] \oplus M_{2 \times 2}(\Z_{2})$.
Furthermore, $DT(\Z_{2}[C_{2}])=0$ and
$D_{12} \simeq S_{3} \times C_{2}$. Hence $\Z_{2}[D_{12}]$ is weakly $N$-hybrid but not $N$-hybrid
where $N$ is the unique normal subgroup of $D_{12}$ of order $3$.
\end{example}

\subsection{Decomposition into $p$-parts}\label{subsec:p-parts}
Let $G$ be a finite group and
let $\mathfrak{A}$ be a $\Z$-order such that $\Z[G] \subseteq \mathfrak{A} \subseteq \Q[G]$.
For each prime $p$, we set $\mathfrak{A}_{p} := \Z_{p} \otimes_{\Z} \mathfrak{A}$.
The canonical maps $K_{0}(\mathfrak{A}, \Q) \rightarrow K_{0}(\mathfrak{A}_{p}, \Q_{p})$
induce isomorphisms
\begin{equation}\label{eq:p-part-decomp}
K_{0}(\mathfrak{A}, \Q) \simeq \bigoplus_{p} K_{0}(\mathfrak{A}_{p}, \Q_{p})
\quad \textrm{ and } \quad
DT(\mathfrak{A}) \simeq \bigoplus_{p} DT(\mathfrak{A}_{p})
\end{equation}
where the sums range over all primes
(see the discussion following \cite[(49.12)]{MR892316}).
We note that for appropriate $\mathfrak{A}$ the maps of \S \ref{subsec:K-func}
respect both these decompositions.

\section{The equivariant Tamagawa number conjecture for Tate motives}\label{sec:ETNC_Tate}

\subsection{Preliminaries and notation}\label{subsec:ETNC-prelim}
We give a very brief description of the statement and properties of the equivariant Tamagawa number conjecture (ETNC) for Tate motives formulated by Burns and Flach \cite{MR1884523}; we
omit all details except those necessary for proofs in later sections.

Let $L/K$ be a finite Galois extension of number fields with Galois group $G$.
For each integer $r$ we set $\Q(r)_{L}:=h^{0}(\mathrm{Spec}(L))(r)$,
which we regard as a motive defined over $K$ and with coefficients in the semisimple algebra $\Q[G]$.
Let  $\mathfrak{A}$ be a $\Z$-order such that $\Z[G] \subseteq \mathfrak{A} \subseteq \Q[G]$.
The conjecture `$\ETNC(\Q(r)_{L}, \mathfrak{A})$' formulated in
\cite[Conjecture 4(iv)]{MR1884523}
for the pair $(\Q(r)_{L}, \mathfrak{A})$ asserts
that a certain canonical element $T\Omega(\Q(r)_{L}, \mathfrak{A})$  of $K_{0}(\mathfrak{A},\R)$ vanishes.
(As observed in \cite[\S 1]{MR1981031}, the element $T\Omega(\Q(r)_{L}, \mathfrak{A})$ is indeed well-defined.)

It will be convenient to explicitly specify the base field, and so
we shall henceforth denote $T\Omega(\Q(r)_{L}, \mathfrak{A})$ by
$T\Omega(L/K,\mathfrak{A},r)$ and say `$\ETNC(L/K,\mathfrak{A},r)$ holds' if this
element vanishes.
If $T\Omega(L/K,\mathfrak{A},r) \in K_{0}(\mathfrak{A},\Q)$
then the first decomposition in \eqref{eq:p-part-decomp} defines an element
$T\Omega(L/K,\mathfrak{A}_{p},r) \in K_{0}(\mathfrak{A}_{p},\Q_{p})$ for each prime $p$,
and we say `$\ETNC(L/K,\mathfrak{A}_{p},r)$' holds if this element vanishes.
We shall abbreviate $\ETNC(L/K,\Z[G],r)$ to $\ETNC(L/K,r)$ and a
subscript $p$ shall have the obvious meaning;
thus $\ETNC(L/K,r)$ holds if and only if $\ETNC_{p}(L/K,r)$ holds for all primes $p$.

\subsection{Functorial properties}\label{subsec:ETNC-funct-props}
We now recall some important functorial properties.
By \cite[Theorem 4.1]{MR1884523} the ETNC behaves well with respect to the maps defined in
\S \ref{subsec:K-func}. In particular, we have the following.

\begin{prop}\label{prop:ETNC-funct}
Let $L/K$ be a finite Galois extension of number fields with Galois group $G$.
Let $r \in \Z$ and suppose that $T\Omega(L/K, \Z[G],r) \in K_{0}(\Z[G],\Q)$.
Let $p$ be a prime.
\begin{enumerate}
\item If $H$ is a subgroup of $G$ then
\[
\res_{H}^{G}(T\Omega(L/K, \Z_{p}[G],r)) = T\Omega(L/L^{H}, \Z_{p}[H], r).
\]
\item If $N$ is a normal subgroup of $G$ then
\[
\quot^{G}_{G/N}(T\Omega(L/K, \Z_{p}[G], r)) = T\Omega(L^{N}/K, \Z_{p}[G/N], r).
\]
\item Let $\mathfrak{A}_{p}$ and $\mathfrak{B}_{p}$ be $\Z_{p}$-orders such that
$\Z_{p}[G] \subseteq \mathfrak{A}_{p} \subseteq \mathfrak{B}_{p} \subseteq \Q_{p}[G]$.
Let $\rho^{\mathfrak{B}_{p}}_{\mathfrak{A}_{p}}$ be the map induced by the extension of scalars functor
$\mathfrak{B}_{p} \otimes_{\mathfrak{A}_{p}} - $.
Then
\[
\rho^{\mathfrak{B}_{p}}_{\mathfrak{A}_{p}}(T\Omega(L/K, \mathfrak{A}_{p}, r))
= T\Omega(L/K, \mathfrak{B}_{p}, r).
\]
\item
Let $\mathfrak{A}_{p}$ be a $\Z_{p}$-order such that $\Z_{p}[G] \subseteq \mathfrak{A}_{p} \subseteq \Q_{p}[G]$
and let $N$ be a normal subgroup of $G$.
Suppose that $e_{N} \in \mathfrak{A}_{p}$
and $e_{N}\mathfrak{A}_{p} = e_{N}\Z_{p}[G] \simeq \Z_{p}[G/N]$. Then
\[
\proj^{G}_{G/N}(T\Omega(L/K, \mathfrak{A}_{p}, r)) =  T\Omega(L^{N}/K, \Z_{p}[G/N], r).
\]
\end{enumerate}
\end{prop}

\begin{proof}
Parts (i) and (ii) follow by taking `$p$-parts' in \cite[Proposition 4.1]{MR1884523}.
For part (iii), let
$\rho:\mathfrak{A}_{p} \hookrightarrow \mathfrak{B}_{p}$ be the natural inclusion and let
$\rho_{*} : K_{0}(\mathfrak{A}_{p},\Q_{p}) \longrightarrow K_{0}(\mathfrak{B}_{p},\Q_{p})$
be the map defined in \S \ref{subsec:K-func}.
Then $\rho^{\mathfrak{A}_{p}}_{\mathfrak{B}_{p}}=\rho_{*}$ and the result follows from
\cite[Theorem 4.1]{MR1884523}.
Part (iv) follows from parts (ii) and (iii) and the observation that $\quot^{G}_{G/N} = \proj^{G}_{G/N} \circ \rho^{\mathfrak{A}_{p}}_{\Z_{p}[G]}$.
\end{proof}

\begin{remark}
 In particular, by Proposition \ref{prop:ETNC-funct}(i) and (ii)
 if $\ETNC_{p}(L/K,r)$ holds then $\ETNC_{p}(F/E,r)$ also holds for every
Galois subextension  $F/E$ of $L/K$. Furthermore, if $\mathfrak{A}_{p}$ and $\mathfrak{B}_{p}$
are as in Proposition \ref{prop:ETNC-funct}(iii) then $\ETNC(L/K,\mathfrak{A}_{p},r)$ implies
 $\ETNC(L/K,\mathfrak{B}_{p},r)$.
\end{remark}

\subsection{The ETNC over group rings}\label{subsec:known-cases}
We list some known cases of $\ETNC(L/K,r)$.

\begin{enumerate}
\item If $L/\Q$ is abelian and $K$ is any subfield of $L$
then $\ETNC(L/K,r)$ holds for any $r \in \Z$. This result is due to Burns and Flach \cite{MR2290586} and builds on work of Burns and Greither \cite{MR1992015}; difficulties with the $2$-part are resolved by Flach \cite{MR2863902}.
\item Let $K$ be an imaginary quadratic field and let $p$ be an odd prime that splits in $K/\Q$ and does not divide the class number of $K$. 
If $L$ is a finite abelian extension of $K$, then $\ETNC_{p}(L/K,0)$ holds. 
This result is due to Bley \cite[Corollary 4.3]{MR2226270} (note that the proof of \cite[Theorem 4.2]{MR2226270} relies on
the main result of \cite{MR797675} on the vanishing of certain $\mu$-invariants, which requires that $p>3$; this result has recently been 
extended in particular to the case $p=3$ by \cite{mu_zero_for_p3}.)
\item Let $K$ be an imaginary quadratic field and let $p$ be an odd prime that splits in $K/\Q$.
If $L$ is a finite abelian extension of $K$, then $\ETNC_{p}(L/K,r)$ holds for any $r<0$.
This result is due to Johnson-Leung \cite[Main Theorem]{Johnson-ETNC} and relies on work of
Johnson-Leung and Kings \cite{MR2794626} (though stated for $p>3$, this result also holds for 
$p=3$ for the same reason as above).
\item If $L$ belongs to a certain infinite family of fields for which $\Gal(L/\Q) \simeq Q_{8}$ then
$\ETNC(L/\Q,r)$ holds for $r=0,1$. These cases are due to Burns and Flach
\cite[Theorem 4.1]{MR1981031} and \cite[(1) and Corollary 1.5]{MR2290586} and rely heavily on results
of Chinburg \cite{MR973358}.
\item If $L/K$ is any quadratic extension then $\ETNC(L/K,0)$ holds.
This is shown by Kim in \cite[\S 2.4, Remark i]{kim-thesis}; also see
\cite[\S 7]{MR3187901}.
\item In \cite[\S 11]{MR3187901}, Buckingham gives examples of biquadratic extensions $L/K$
(with $K/\Q$ imaginary quadratic and $L/\Q$ non-abelian) for which $\ETNC(L/K,0)$ holds.
\item In \cite{MR2268756}, Navilarekallu verifies $\ETNC(L/\Q,0)$ for a particular field $L$ with $\Gal(L/\Q) \simeq A_{4}$.
By Proposition \ref{prop:ETNC-funct}(i), this gives four intermediate fields $F$ with $\Gal(L/F) \simeq C_{3}$
and $F/\Q$ quartic and non-Galois such that $\ETNC(L/F,0)$ holds.
\item If $L/K$ is a Galois CM-extension and $p$ is an odd prime, then $\ETNC_{p}(L/K,0)$ naturally decomposes into
a plus and a minus part. Under certain restrictions, the second named author deduces
the minus part of $\ETNC_{p}(L/K,0)$  from the conjectural vanishing of certain $\mu$-invariants
 \cite{MR2822866, stickelberger}.
\item If $L/K$ is a finite Galois extension of totally real number fields and $r<0$ is odd,
then $\ETNC_{p}(L/K,r)$ holds for any odd prime $p$, provided that certain $\mu$-invariants vanish.
A similar result holds on minus parts if $L/K$ is a CM-extension and $r<0$ is even.
These results are due to Burns \cite[Corollary 2.10]{burns-mc}.
\end{enumerate}

\subsection{The ETNC over maximal orders}\label{subsec:ETNC-max}
Let $L/K$ be a Galois extension of number fields with Galois group $G$.
For a maximal $\Z$-order $\mathfrak{M}(G)$ with $\Z[G] \subseteq \mathfrak{M}(G) \subseteq \Q[G]$
we shall abbreviate
$\ETNC(L/K,\mathfrak{M}(G),r)$ to $\ETNC^{\max}(L/K,r)$; this is independent of
the choice of $\mathfrak{M}(G)$.
We define $\ETNC_{p}^{\max}(L/K,r)$ analogously.
We list some properties and known cases of
$\ETNC^{\max}(L/K,r)$ and $\ETNC_{p}^{\max}(L/K,r)$.

\begin{enumerate}
\item If $p \nmid |G|$ then $\Z_{p}[G]$ is a maximal $\Z_{p}$-order and so the
statements $\ETNC_{p}(L/K,r)$ and $\ETNC_{p}^{\max}(L/K,r)$ are equivalent in this case.
\item By Proposition \ref{prop:torsion-kernel} and Proposition \ref{prop:ETNC-funct}(iii),
$\ETNC_{p}^{\max}(L/K,r)$ is equivalent to $T\Omega(L/K,\Z_{p}[G],r) \in DT(\Z_{p}[G])$.
Thus $\ETNC_{p}(L/K,r)$ implies $\ETNC_{p}^{\max}(L/K,r)$, and so the latter holds in all
the cases listed in \S \ref{subsec:known-cases}.
\item Burns and Flach \cite[\S 3, Corollary 1]{MR1981031}
show that $\ETNC^{\max}(L/K,0)$ is equivalent
to the strong Stark conjecture (as formulated by Chinburg \cite[Conjecture 2.2]{MR724009})
for $L/K$. Thus we write $\SSC(L/K)$ for $\ETNC^{\max}(L/K,0)$ and
$\SSC_{p}(L/K)$ for $\ETNC^{\max}_{p}(L/K,0)$.
\item Let $K$ be an imaginary quadratic field and $L$ be a finite abelian extension of $K$.
Let $p$ be a prime such that either
$[L:K]$ is a power of $p$ or $p$ does not divide the class number of $K$.
Then $\SSC_{p}(L/K)$ holds by work of Bley \cite[Theorem 1.1 and Corollary 1.2]{werner-hab}.
\item $\SSC(L/K)$ can be broken
down into $\chi$-parts $\SSC(L/K)(\chi)$ where $\chi \in \Irr_{\C}(G)$.
If $\chi$ is rational-valued then $\SSC(L/K)(\chi)$ holds by a result of Tate
\cite[Chapter II, Theorem 6.8]{MR782485}.
If $L^{\ker(\chi)}/\Q$ is abelian then $\SSC(L/K)(\chi)$
holds by combining \S \ref{subsec:known-cases}(i) and \S \ref{subsec:ETNC-max}(iii);
outside the $2$-part this result was first established by Ritter and Weiss \cite[Theorem A]{MR1423032}.
\item If $L/K$ is a Galois extension of totally real number fields and $r<0$ is odd,
then $\ETNC^{\max}_{p}(L/K, r)$ holds for every odd prime $p$. This is a result of the second named author \cite[Corollary 6.2]{MR2801311}.
\item A result similar to (vi) holds on minus parts if $L/K$ is a CM-extension and $r<0$ is even.
This follows if one combines a result of Burns \cite[Corollary 2.10]{burns-mc} with a general induction argument of the second named author  \cite[Proposition 6.1(iii)]{MR2801311}.
\end{enumerate}

\subsection{The ETNC over weakly hybrid orders}
We show how weakly hybrid orders can be used to `break up' certain cases of the ETNC.

\begin{theorem}\label{thm:ETNC-break-down}
Let $L/K$ be a finite Galois extension of number fields with Galois group $G$ and let $r \in \Z$.
Suppose that $T\Omega(L/K, \Z[G],r) \in K_{0}(\Z[G],\Q)$.
Let $p$ be a prime and let $\mathfrak{A}_{p}$ be a $\Z_{p}$-order such that
$\Z_{p}[G] \subseteq \mathfrak{A}_{p} \subseteq \Q_{p}[G]$.
Suppose that $\mathfrak{A}_{p}$ is weakly $N$-hybrid.
Then
$\ETNC(L/K, \mathfrak{A}_{p}, r)$ holds if and only if both $\ETNC_{p}(L^{N}/K,r)$ and
$\ETNC^{\max}_{p}(L/K,r)$ hold.
\end{theorem}

\begin{proof}
Suppose $\ETNC_{p}^{\max}(L/K, r)$ holds. Then $T\Omega(L/K,\Z_{p}[G],r) \in DT(\Z_{p}[G])$
by \S \ref{subsec:ETNC-max}(ii). If we now further assume that $\ETNC_{p}(L^{N}/K,r)$ holds,
it follows from Proposition \ref{prop:wh-quot} and
Proposition \ref{prop:ETNC-funct}(iv) that $\ETNC(L/K, \mathfrak{A}_{p},r)$ holds.

Suppose conversely that $\ETNC(L/K, \mathfrak{A}_{p},r)$ holds.
Then $\ETNC_{p}(L^{N}/K,r)$ holds by Proposition \ref{prop:ETNC-funct}(iv)
and  $\ETNC_{p}^{\max}(L/K, r)$ holds by applying Proposition \ref{prop:ETNC-funct}(iii)
with $\mathfrak{B}_{p}$ equal to some maximal $\Z_{p}$-order
such that $\mathfrak{A}_{p} \subseteq \mathfrak{B}_{p} \subseteq \Q_{p}[G]$.
\end{proof}

\begin{corollary}\label{cor:ETNC-break-down}
Assume the situation and notation of Theorem \ref{thm:ETNC-break-down}.
Further suppose that $N=G'$ is the commutator subgroup of $G$ and that $L^{N}/\Q$ is abelian
(in particular, this is the case when $K=\Q$).
Then $\ETNC(L/K, \mathfrak{A}_{p}, r)$ holds if and only if $\ETNC^{\max}_{p}(L/K,r)$ holds.
\end{corollary}

\begin{proof}
This follows by combining Theorem \ref{thm:ETNC-break-down} and
\S \ref{subsec:known-cases}(i).
\end{proof}

We end this subsection with the following observation.

\begin{prop} \label{prop:ETNC2-reduction}
Let $A$ be a finite abelian group of odd order and let $C_{2}$ act on $A$ by inversion.
Put $G = A \rtimes C_{2}$ as in Example \ref{ex:inversion} and let $r \in \Z$.
Let $L/K$ be a Galois extension of number fields with $\Gal(L/K) \simeq G$.
Suppose that $T\Omega(L/K, \Z[G],r) \in K_{0}(\Z[G],\Q)$.
Then $\ETNC_{2}(L/K,r)$ holds if and only if $\ETNC^{\max}_{2}(L/K,r)$ holds.
\end{prop}

\begin{proof}
This is immediate, as $DT(\Z_{2}[G])$ is trivial by Lemma \ref{lem:trivial-DT2}.
\end{proof}

\subsection{The case $r=0$}\label{subsec:case-r=0}
We now establish new cases of $\ETNC(L/K,0)$ and $\ETNC_{p}(L/K,0)$.
A brief discussion of the relation to other conjectures is given at the beginning of
\S \ref{sec:other-conjectures}.

\begin{theorem}\label{thm:affine-ETNC}
Let $q$ be a prime power. Let $G=\Aff(q)$ be the Frobenius group of order $q(q-1)$ defined in Example \ref{ex:affine}
and let $N$ be its Frobenius kernel.
Let $L/K$ be a finite Galois extension of number fields with $\Gal(L/K) \simeq G$.
Suppose that $L^{N}/\Q$ is abelian (in particular, this is the case when $K=\Q$).
Then $\SSC(L/K)$ holds and
$\ETNC_{p}(L/K, 0)$ holds for all primes $p$ not dividing $q$.
\end{theorem}

\begin{proof}
The only non-linear irreducible character of $G$ is rational-valued and all linear characters factor
through $\Gal(L^{N}/K)$, so $\SSC(L/K)$ holds by \S \ref{subsec:ETNC-max}(v).
By Example \ref{ex:affine} the group ring $\Z_{p}[G]$ is $N$-hybrid for all primes $p$ not dividing $q$.
The result now follows by applying Corollary \ref{cor:ETNC-break-down}.
\end{proof}

\begin{remark}
Up until now there has been no known example
of a finite non-abelian group $G$ and an odd prime $p$ dividing the order of $G$
such that $\ETNC_{p}(L/\Q,0)$ has been shown to hold for every extension
$L/\Q$ with $\Gal(L/\Q) \simeq G$; Theorem \ref{thm:affine-ETNC} gives infinitely many such examples.
(Note that the condition that $p$ be odd here is due to Proposition \ref{prop:ETNC-S3}, the result of
which is well-known.)
\end{remark}

\begin{corollary}\label{cor:rel-cyclic-ETNC}
Assume the situation and notation of Theorem \ref{thm:affine-ETNC}.
Let $H$ be any subgroup of any choice of Frobenius complement in $G$ (hence $H$ is isomorphic to a subgroup
of $\F_{q}^{\times} \simeq C_{q-1}$). Then $\ETNC(L/L^{H},0)$ holds.
\end{corollary}

\begin{proof}
By Proposition \ref{prop:ETNC-funct}(i) we see that $\ETNC_{p}(L/L^{H},0)$ holds for all primes $p$
not dividing $q$. Now suppose $p$ divides $q$. Then as $\SSC(L/K)$ holds
$T\Omega(L/K, \Z_{p}[G],0) \in DT(\Z_{p}[G])$ by \S \ref{subsec:ETNC-max}(ii).
Thus by Proposition \ref{prop:ETNC-funct}(i) $T\Omega(L/L^{H}, \Z_{p}[H],0) \in DT(\Z_{p}[H])$.
But $DT(\Z_{p}[H])$ is trivial as $p \nmid |H|$ and so $\Z_{p}[H]$ is maximal
(see Proposition \ref{prop:torsion-units}).
\end{proof}

\begin{remark}\label{rmk:inf-many-rel-cyclic}
Fix a natural number $n$.
By using \v{S}aferevi\v{c}'s Theorem on the realisability of soluble groups as Galois groups over global fields
(see the account in \cite[Chapter IX, \S 6]{MR2392026}) and
Dirichlet's theorem on primes in arithmetic progressions together
with Corollary \ref{cor:rel-cyclic-ETNC},
we see that there is an infinite family of Galois extensions of number fields $L/F$ with $\Gal(L/F) \simeq C_{n}$
and $F/\Q$ non-abelian (indeed non-Galois) such that $\ETNC(L/F,0)$ holds.
To date the only examples $L/F$ with $F/\Q$ non-abelian for which $\ETNC(L/F,0)$
is known to hold have been either trivial, quadratic, or cubic (see \S \ref{subsec:known-cases}).
\end{remark}

\begin{remark}
Let $G$ be the Frobenius group of order $72$ described in Example \ref{ex:frob72}.
Let $N$ be the Frobenius kernel of $G$ and let $H$ be some choice of Frobenius complement.
Let $L/K$ be a Galois extension of number fields with $\Gal(L/K) \simeq G$.
As every complex irreducible character of $G$ is rational-valued, $\SSC(L/K)$ holds by  \S \ref{subsec:ETNC-max}(v).
Furthermore Theorem \ref{thm:ETNC-break-down} shows that $\ETNC_{2}(L/K,0)$ holds if and only if
$\ETNC_{2}(L^{N}/K,0)$ holds.

Note that $\Gal(L^{N}/K) \simeq H \simeq Q_{8}$.
Recall from \S \ref{subsec:known-cases}(iv) that there is an infinite family of extensions $F/\Q$ such that
$\Gal(F/\Q) \simeq Q_{8}$ and $\ETNC(F/\Q,0)$ holds.
Thus if the appropriate Galois embedding problem can be solved, one can give examples of extensions $L/\Q$
with $\Gal(L/\Q) \simeq G$ such that $\ETNC_{2}(L/\Q,0)$ and $\SSC(L/\Q)$ both hold.
Essentially the same argument can be given when $G$ is replaced by any member of the infinite family
of Frobenius groups given in \cite{MR2104937}.
\end{remark}

\begin{prop}\label{prop:res-D2n-Cn}
Let $n$ be an odd integer and let $p$ be an odd prime. Then
\[
\res^{D_{2n}}_{C_{n}}: DT(\Z_{p}[D_{2n}]) \longrightarrow DT(\Z_{p}[C_{n}])
\]
is injective (note that there is a unique subgroup of $D_{2n}$ isomorphic to $C_{n}$).
\end{prop}

\begin{proof}
This is \cite[Proposition 3.2(2)]{MR2031413}.
\end{proof}

We now combine results of Bley \cite[Corollary 4.3]{MR2226270} (see \S \ref{subsec:known-cases})
and of Breuning (Proposition \ref{prop:res-D2n-Cn})
to prove the following result.

\begin{theorem}\label{thm:ETNC-dihedral}
Let $L/\Q$ be a Galois extension with $\Gal(L/\Q) \simeq D_{2n}$ for some odd $n$.
Let $K/\Q$ be the unique quadratic subextension of $L/\Q$ and suppose that $K$ is imaginary.
Let $p$ be a prime and suppose that $p$ does not divide the class number of $K$.
If $p$ divides $n$, further suppose that $p$ is odd and splits in $K/\Q$.
Then $\ETNC_{p}(L/\Q,0)$ holds.
\end{theorem}

\begin{proof}
By \S \ref{subsec:ETNC-max}(iv) $\SSC_{p}(L/K)$ holds.
There are two linear characters of $\Gal(L/\Q)$,
both of which are rational-valued; hence $\SSC(L/\Q)(\chi)$
holds for these characters by \S \ref{subsec:ETNC-max}(v).
Since $\Gal(L/\Q) \simeq D_{2n}$ is a Frobenius group,
all non-linear irreducible characters are induced from nontrivial characters of $\Gal(L/K)$
by Theorem \ref{thm:frob-kernel}(iv);
hence $\SSC_{p}(L/\Q)(\chi)$ holds for these characters by \cite[Proposition 9(c)]{MR1423032}.
Thus we have established $\SSC_{p}(L/\Q)$.
Hence by \S \ref{subsec:ETNC-max}(ii) we have
$T\Omega(L/\Q,\Z_{p}[G],0) \in DT(\Z_{p}[G])$.

By \S \ref{subsec:ETNC-max}(i) it remains to verify $\ETNC_{p}(L/\Q,0)$ in the case that $p$ divides $2n$.
By Lemma \ref{lem:trivial-DT2} (or by \cite[Proposition 3.2(1)]{MR2031413}) $DT(\Z_{2}[D_{2n}])$ is trivial,
giving the $p=2$ case.
Now suppose that $p$ is odd and divides $n$.
Then by \S \ref{subsec:known-cases}(ii) $\ETNC_{p}(L/K,0)$ holds.
Now by Proposition \ref{prop:res-D2n-Cn} the restriction map $DT(\Z_{p}[\Gal(L/\Q)]) \rightarrow DT(\Z_{p}[\Gal(L/K)])$ is injective.
Hence the desired result now follows by Proposition \ref{prop:ETNC-funct}(i).
\end{proof}

\begin{corollary}\label{cor:ETNC-dihedral}
Assume the setting and notation of Theorem \ref{thm:ETNC-dihedral}.
Suppose that $K$ has class number $1$, $n$ is odd,
and every prime $p$ dividing $n$ splits in $K/\Q$.
Then  $\ETNC(L/\Q,0)$ holds.
\end{corollary}

\begin{remark}\label{rmk:infinite-ETNC-family}
Let $K/\Q$ be a quadratic extension and let $p$ be a prime.
Then \cite[Theorem I.2.1]{MR680538} says that there are infinitely many fields $L$ such that
$\Gal(L/\Q) \simeq D_{2p}$ and $K \subseteq L$. (The authors are grateful to James Newton for providing the reference for this result.)
Now let $K=\Q(i)$ and fix a prime $p$ such that $p \equiv 1 \pmod 4$.
Then $K$ has class number $1$ and $p$ splits in $K/\Q$.
Therefore by Corollary \ref{cor:ETNC-dihedral} there are infinitely many extensions $L/\Q$ with $\Gal(L/\Q) \simeq D_{2p}$
such that $K \subseteq L$ and $\ETNC(L/\Q,0)$ holds.
(Clearly it is possible to obtain similar results by replacing $K$ with other imaginary quadratic fields with class number $1$.)
Since there are infinitely many primes $p \equiv 1 \pmod 4$,
we have given an explicit infinite family of finite non-abelian groups with the property that for each member $G$
there are infinitely many extensions $L/\Q$ with $\Gal(L/\Q) \simeq G$ such that $\ETNC(L/\Q,0)$ holds.
Up until now, the only known family of finite non-abelian groups with this property has been that containing the single group $Q_{8}$
(see \S \ref{subsec:known-cases}(iv)).
\end{remark}

We now make the following observations, the first of which is well-known.

\begin{prop}\label{prop:ETNC-S3}
Let $L/K$ be a Galois extension of number fields with $\Gal(L/K) \simeq S_{3}$.
Then $\SSC(L/K)$ holds and $\ETNC_p(L/K,0)$ is true for every prime $p \neq 3$.
\end{prop}

\begin{proof}
As all characters of $S_{3}$ are rational-valued, $\SSC(L/K)$ holds by \S \ref{subsec:ETNC-max}(v).
This also implies $\ETNC_{p}(L/K,0)$ for every $p > 3$.
The case $p=2$ is a consequence of Proposition \ref{prop:ETNC2-reduction} with $A=C_{3}$.
\end{proof}

\begin{prop} \label{prop:ETNC-D12}
Let $L/K$ be a Galois extension of number fields with $\Gal(L/K) \simeq D_{12}$.
Let $N\simeq C_{3}$ be the commutator subgroup and suppose that $L^{N}/\Q$ is abelian
(in particular, this is the case when $K = \Q$).
Then $\SSC(L/K)$ holds and $\ETNC_{p}(L/K,0)$ is true for every prime $p \neq 3$.
\end{prop}

\begin{proof}
As all characters of $D_{12} \simeq S_{3} \times C_{2}$ are rational-valued, $\SSC(L/K)$ holds by \S \ref{subsec:ETNC-max}(v).
This also implies $\ETNC_{p}(L/K,0)$ for every $p > 3$.
Furthermore $\Z_{2}[D_{12}]$ is weakly $N$-hybrid by Example \ref{ex:D12-weakly-hybrid}, and so $\ETNC_{2}(L/K,0)$ follows from Corollary \ref{cor:ETNC-break-down}.
\end{proof}

\begin{prop}\label{prop:ETNC-S3-S4}
Let $L/K$ be a Galois extension of number fields with $\Gal(L/K) \simeq S_{4}$.
Then $\SSC(L/K)$ holds and $\ETNC_{3}(L/K,0)$ holds if and only if $\ETNC_{3}(L^{V_{4}}/K,0)$ holds.	
(Note that $S_{4}/V_{4} \simeq S_{3}$ and so $\Gal(L^{V_{4}}/K) \simeq S_{3}$.)
\end{prop}

\begin{proof}
As all characters of $S_{4}$ are rational-valued, $\SSC(L/K)$ holds by \S \ref{subsec:ETNC-max}(v).
The remaining claim now follows by combining Example \ref{ex:S4-V4} and Theorem \ref{thm:ETNC-break-down}.
\end{proof}

We now combine the above partial results to give the following theorem.

\begin{theorem}
Let $G$ be a finite group containing a subgroup $H$ such that $[G:H]=2$.
Let $L/\Q$ be a finite Galois extension with $\Gal(L/\Q) \simeq G$ and let $K=L^{H}$.
Suppose that $K$ is imaginary, $3$ splits in $K/\Q$, and $3$ does not divide the class number of $K$.
\begin{enumerate}
\item If $G = S_{3}$ and $H=A_{3}$ then $\ETNC(L/\Q,0)$ holds.
\item If $G = D_{12} = S_{3} \times C_{2}$ and $H = A_{3} \times C_{2}$ then $\ETNC(L/\Q,0)$ holds.
\item If $G=S_{4}$ and $H=A_{4}$ then  $\ETNC(L/\Q,0)$ holds outside the $2$-part.
\item If $G=S_{4} \times C_{2}$ and $H=A_{4} \times C_{2}$ then  $\ETNC(L/\Q,0)$ holds outside the $2$-part.
\end{enumerate}
\end{theorem}

\begin{proof}
Case (i) follows from Proposition \ref{prop:ETNC-S3} and Theorem \ref{thm:ETNC-dihedral} with $p=3$.
Similarly, case (ii) follows from Proposition \ref{prop:ETNC-D12} and Theorem \ref{thm:ETNC-dihedral} with $p=3$.
Case (iii) follows from case (i) and Proposition \ref{prop:ETNC-S3-S4}. 
Finally, case (iv) follows from case (ii) and an analogous version of Proposition \ref{prop:ETNC-S3-S4} 
that uses the fact that $\Z_{3}[S_{4} \times C_{2}]$ is $V_{4} \times \{ 1 \}$-hybrid 
(to see this, apply Lemma \ref{lem:add-a-group} to Example \ref{ex:S4-V4}).
\end{proof}

\subsection{The case $r<0$} \label{subsec:case-r<0}
We now establish new cases of $\ETNC(L/K,r)$ and $\ETNC_{p}(L/K,r)$ in the case $r<0$.

\begin{theorem}\label{thm:hybrid-tot-real-r-less-than-zero}
Let $L/K$ be a finite Galois extension of totally real number fields with Galois group $G$.
Let $p$ be an odd prime and let $N=G'$ be the commutator subgroup of $G$.
Suppose that $\Z_{p}[G]$ is weakly $N$-hybrid and that $L^{N} / \Q$ is abelian
(in particular, this is the case when $K=\Q$).
Then $\ETNC_{p}(L/K,r)$ holds for every odd $r<0$.
\end{theorem}

\begin{proof}
This is just the combination of Corollary \ref{cor:ETNC-break-down} and \S \ref{subsec:ETNC-max}(vi).
\end{proof}

\begin{theorem} \label{thm:ETNC-Frobenius-groups}
Let $G = N \rtimes H$ be a Frobenius group.
Suppose that $L/K$ is a finite Galois extension of totally real number fields with $\Gal(L/K) \simeq G$ and that 
$L^{N}/\Q$ is abelian (in particular, this is the case when $K=\Q$ and $H$ is abelian).
Then $\ETNC_{p}(L/K,r)$ holds for every odd $r<0$ and every prime $p \nmid 2|N|$.
If in addition $N$ is an $\ell$-group for a prime $\ell$, then $\ETNC(L/K,r)$ holds outside its $2$-part for every odd $r < 0$.
In particular, this applies in the following cases:
\begin{itemize}
    \item
    $G \simeq \Aff(q)$, where $q$ is a prime power (see Example \ref{ex:affine}).
    \item
    $G \simeq P \rtimes C_{2}$, where $P$ is an abelian $\ell$-group  (with $\ell$ odd) on which $C_{2}$ acts by inversion
    (see Example \ref{ex:inversion}).
    \item
    $G \simeq C_{\ell} \rtimes C_{p}$, where $p<\ell$ are distinct primes such that $p \mid (\ell-1)$ and $C_{p}$ acts on $C_{\ell}$ via an embedding $C_{p} \hookrightarrow \Aut(C_{\ell})$ (see Example \ref{ex:metacyclic}).
\end{itemize}
\end{theorem}

\begin{proof}
As $\Z_{p}[G]$ is $N$-hybrid for every prime $p \nmid |N|$ by Proposition \ref{prop:frob-N-hybrid},
the first claim follows immediately from Theorem \ref{thm:hybrid-tot-real-r-less-than-zero}.
Now assume that $N$ is an $\ell$-group. If $\ell=2$, there is nothing to show.
If $\ell$ is odd, the extension $L/L^{N}$ has $\ell$-power degree and $L^{N}/\Q$ is abelian by assumption.
Then $\ETNC_{\ell}(L/K,r)$ follows from \cite[Corollary 2.10]{burns-mc},
since the $\mu$-invariant occurring therein vanishes by a result of Ferrero and Washington \cite{MR528968}.
\end{proof}

The same arguments as in the proof of Corollary \ref{cor:rel-cyclic-ETNC} lead to the following result.

\begin{corollary}
Let $L$ be a totally real Galois extension of $\Q$ with Galois group $G$.
Assume that $G$ is a Frobenius group with abelian Frobenius complement.
Let $H$ be any subgroup of any choice of Frobenius complement in $G$.
Then $\ETNC(L/L^H,r)$ holds outside its $2$-part for every odd $r<0$.
\end{corollary}

\begin{remark}
Fix a natural number $n$. The same reasoning used in Remark \ref{rmk:inf-many-rel-cyclic} shows that there exists
an infinite family of Galois extensions of number fields $L/F$ with $\Gal(L/F) \simeq C_{n}$ such that
$F/\Q$ is non-Galois and $\ETNC(L/F,r)$ holds outside its $2$-part for every odd integer $r<0$.
\end{remark}

\begin{theorem}\label{thm:ETNC-S4-negative-r}
Let $L/K$ be a finite Galois extension of totally real number fields with $\Gal(L/K) \simeq S_{4}$.
Suppose that $L^{A_{4}} / \Q$ is abelian (in particular, this is the case when $K=\Q$).
Then $\ETNC(L/K,r)$ holds outside its $2$-part for every odd $r<0$.
\end{theorem}

\begin{proof}
By \S \ref{subsec:ETNC-max}(vi) we know that $\ETNC^{\max}_{p}(L/K,r)$ holds for every odd prime $p$.
Thus it only remains to verify $\ETNC_{3}(L/K,r)$.
By combining Example \ref{ex:S4-V4} and Theorem \ref{thm:ETNC-break-down}, we are reduced
to showing $\ETNC_{3}(L'/K,r)$, where $L' = L^{V_{4}}$.
As $\Gal(L'/K) \simeq S_{3}$ and $L^{A_{4}} = (L')^{A_{3}}$,
the latter is a special case of Theorem \ref{thm:ETNC-Frobenius-groups}.
\end{proof}

\begin{remark}
In Theorems \ref{thm:hybrid-tot-real-r-less-than-zero} and \ref{thm:ETNC-Frobenius-groups} and their corollaries, and in Theorem \ref{thm:ETNC-S4-negative-r},
we have considered extensions of totally real fields and odd integers $r<0$.
However, we note that all these results have analogues for CM-extensions and even integers $r<0$:
one simply has to restrict to minus parts and replace each occurrence of a Galois group $G$ by $G \times C_{2}$, where $C_{2}$ is generated by complex conjugation.
Although, as mentioned in \S \ref{subsec:known-cases}(ix), all these results follow from the conjectural vanishing of certain $\mu$-invariants as established by Burns \cite[Corollary 2.10]{burns-mc}, the advantage of our approach is that it leads to unconditional results.
\end{remark}

\begin{theorem}\label{thm:ETNC-dihedral<0}
Let $L/\Q$ be a Galois extension with $\Gal(L/\Q) \simeq D_{2n}$ for some odd $n$.
Let $K/\Q$ be the unique quadratic subextension of $L/\Q$ and suppose that $K$ is imaginary.
Let $p$ be an odd prime that splits in $K/\Q$.
Then $\ETNC_{p}(L/\Q,r)$ holds for every integer $r<0$.
\end{theorem}

\begin{proof}
This is essentially the same proof as that of Theorem \ref{thm:ETNC-dihedral}, but now we use \S \ref{subsec:known-cases}(iii) rather than \S \ref{subsec:known-cases}(ii) and \S \ref{subsec:ETNC-max}(iv).
\end{proof}

\section{Other conjectures on vanishing of elements in relative $K$-groups}\label{sec:other-conjectures}

The results of \S \ref{sec:rel-k-group-weak-hybrid} can be applied to other conjectures concerning
the vanishing of certain elements in relative algebraic $K$-groups, provided these elements satisfy
certain functorial properties.
In principle, new results in the style of \S \ref{sec:ETNC_Tate} can be obtained in this way, though this
depends on an adequate supply of existing results on which to build.
Some of the conjectures which may be considered include the following.

\begin{enumerate}
\item The equivariant Tamagawa number conjecture as formulated in \cite[Conjecture 4(iv)]{MR1884523}
(in \S \ref{sec:ETNC_Tate} we considered the special case of Tate motives).
\item The global equivariant epsilon constant conjecture of \cite{MR2005875}.
\item The local equivariant epsilon constant conjecture of \cite{MR2078894}.
\item The leading term conjecture at $s=0$ of \cite[\S 4]{MR2371375}.
\item The leading term conjecture at $s=1$ of \cite[\S 3]{MR2371375}.
\end{enumerate}

Note that the Lifted Root Number Conjecture of Gruenberg, Ritter and Weiss \cite{MR1687551}, the vanishing of $T\Omega(L/K,0)$
defined in \cite[\S 2.1]{MR1863302}, the relevant special case of the ETNC, and (iv) above are all
equivalent (see \cite[Theorems 2.3.3 and 2.4.1]{MR1863302} and \cite[Remarks 4.3 and 4.5]{MR2371375}).
In \cite{MR2804251} it is shown that assuming Leopoldt's conjecture, (v) is equivalent to
the relevant special case of the ETNC. The relation between (ii) and (iii) is explained in
\S \ref{subsec:gec} (see \cite[\S 4]{MR2078894} for more details).
In \cite[\S 5]{MR2371375} it is shown that the compatibility of (iv) and (v) is equivalent to (ii).

\subsection{The Local Epsilon Constant Conjecture}\label{subsec:lec}
We examine applications to the conjecture formulated by Breuning \cite{MR2078894}.
Let $p$ be a prime and let $\Q_{p}^{c}$ be an algebraic closure of $\Q_{p}$.
Let $L/K$ be a finite Galois extension of $p$-adic fields with Galois group $G$.
An element $R_{L/K} \in K_{0}(\Z_{p}[G],\Q_{p}^{c})$ incorporating local epsilon constants
and algebraic invariants associated to the extension $L/K$ is defined.
The conjecture is that $R_{L/K}$ always vanishes.
We recall the following properties.

\begin{theorem}[\cite{MR2078894}]\label{thm:lec-properties}
\hspace*{\fill}
\begin{enumerate}
\item If $N$ is a normal subgroup of $G$ then $\quot_{G/N}^{G}(R_{L/K})=R_{L^{N}/K}$.
\item If $L/K$ is at most tamely ramified then $R_{L/K}=0$.
\item We always have $R_{L/K} \in DT(\Z_{p}[G])$.
\end{enumerate}
\end{theorem}

Further work concerning the conjecture can be found in
\cite{MR2078894}, \cite{MR2546007} and \cite{MR3073206}.
We now make the following straightforward observation.

\begin{prop}
Let $A$ be a nontrivial finite abelian group of odd order and let $C_{2}$ act on $A$ by inversion.
Let $G=A \rtimes C_{2}$ as in Example \ref{ex:inversion}.
Then the local epsilon constant conjecture holds
for every finite Galois extension $L/K$ of $2$-adic fields with $\Gal(L/K) \simeq G$.
\end{prop}

\begin{proof}
We have $R_{L/K} \in DT(\Z_{2}[G])$ by Theorem \ref{thm:lec-properties}(iii).
However, Lemma \ref{lem:trivial-DT2} shows that $DT(\Z_{2}[G])$ is trivial in this case.
\end{proof}

\begin{prop}\label{prop:local-epsilon-hybrid}
Fix a prime $p$ and let $L/K$ be a finite Galois extension of $p$-adic fields with Galois group $G$.
Suppose that $\Z_{p}[G]$ is weakly $N$-hybrid.
Then the local epsilon constant conjecture holds for $L/K$ if and only if it holds for $L^{N}/K$.
\end{prop}

\begin{proof}
This follows by combining Theorem \ref{thm:lec-properties}(i) and (iii) with
Proposition \ref{prop:wh-quot}.
\end{proof}

\begin{remark}\label{rmk:lec-frobenius}
Fix a prime $p$ and let $L/K$ be a finite Galois extension of $p$-adic fields with Galois group $G$.
Suppose that $G$ is a Frobenius group with Frobenius kernel $N$ and that $p \nmid |N|$.
Then $\Z_{p}[G]$ is $N$-hybrid by Proposition \ref{prop:frob-N-hybrid} and so the local epsilon constant conjecture holds
for $L/K$ if and only if it holds for $L^{N}/K$.
Let $G_{1} \unlhd G$ be the wild inertia subgroup of $G$.
Then by Theorem \ref{thm:frob-kernel}(iii) either $G_{1}  \unlhd N$ or $N  \unlhd G_{1}$.
Since $|G_{1}|$ is a power of $p$ and $p \nmid |N|$, we see that $G_{1}$ is trivial.
Thus, $L/K$ is at most tamely ramified.
However, the local epsilon constant conjecture is already known in this case by Theorem \ref{thm:lec-properties}(ii)
and so Proposition \ref{prop:local-epsilon-hybrid} does not provide new information in this situation.
\end{remark}

\begin{remark}\label{rmk:lec-S4-V4}
In Example \ref{ex:S4-V4} it was shown that $\Z_{3}[S_{4}]$ is $V_{4}$-hybrid but that $S_{4}$ is not a
Frobenius group.
One may be tempted to use Proposition \ref{prop:local-epsilon-hybrid} to deduce new cases of
the local epsilon conjecture for extensions $L/K$ of $3$-adic fields with $\Gal(L/K) \simeq S_{4}$
from wildly ramified subextensions $F/K$ with $\Gal(F/K) \simeq S_{4}/V_{4} \simeq S_{3}$.
However, some care is needed: there are no extensions $L/K$ of $3$-adic fields with $\Gal(L/K) \simeq S_{4}$
even though there are extensions $F/K$ with $\Gal(F/K) \simeq S_{3}$
(see \cite{MR2322733}, for example).
\end{remark}

\begin{prop}\label{prop:lecc-D12}
The local epsilon constant conjecture holds for all Galois extensions $L/\Q_{2}$ with $\Gal(L/\Q_{2}) \simeq	D_{12}$.
\end{prop}

\begin{remark}
Note that extensions $L/\Q_{2}$ with $\Gal(L/\Q_{2}) \simeq D_{12}$ do indeed exist (see \cite{MR2194887}, for example) and are necessarily wildly ramified. However, this particular case of the conjecture
has recently been established by Bley and Debeerst in \cite[Theorem 1(b)]{MR3073206} using computational methods,
and our proof relies on a different case of their results.
\end{remark}

\begin{proof}[Proof of Proposition \ref{prop:lecc-D12}]
In Example \ref{ex:D12-weakly-hybrid} it was shown that $\Z_{2}[D_{12}]$ is weakly $N$-hybrid where $N$
is the unique normal subgroup of $D_{12}$ of order $3$.
By \cite[Theorem 1(b)]{MR3073206}, the local epsilon constant conjecture is known for all biquadratic extensions
of $\Q_{2}$.
Thus noting that $D_{12}/N \simeq C_{2} \times C_{2}$, Proposition \ref{prop:local-epsilon-hybrid} shows that
the local epsilon constant conjecture holds for $L/\Q_{2}$.
\end{proof}

\subsection{The Global Epsilon Constant Conjecture}\label{subsec:gec}
We examine applications to the conjecture formulated by Bley and Burns in \cite{MR2005875}.
Let $L/K$ be a finite Galois extension of number fields with Galois group $G$.
An element $T\Omega^{\loc}(L/K,1) \in K_{0}(\Z[G],\R)$ is defined and it is shown that
in fact we always have $T\Omega^{\loc}(L/K,1) \in K_{0}(\Z[G],\Q)$.
The conjecture is that this element always vanishes.
The decomposition $K_{0}(\Z[G],\Q) \simeq \oplus_{p} K_{0}(\Z_{p}[G],\Q_{p})$
(see \S \ref{subsec:p-parts}) splits $T\Omega^{\loc}(L/K,1)$ into $p$-parts
 $T\Omega^{\loc}_{p}(L/K,1) \in K_{0}(\Z_{p}[G],\Q_{p})$.
 We recall the following properties.

\begin{theorem}[\cite{MR2005875}]\label{thm:gec-properties}
Let $p$ be a prime.
\begin{enumerate}
\item If $N$ is a normal subgroup of $G$ then
$\quot_{G/N}^{G}(T\Omega^{\loc}_{p}(L/K,1))=T\Omega^{\loc}_{p}(L^{N}/K,1)$.
\item If $L/K$ is at most tamely ramified then $T\Omega^{\loc}_{p}(L/K,1)=0$.
\item We always have $T\Omega^{\loc}_{p}(L/K,1) \in DT(\Z_{p}[G])$.
\end{enumerate}
\end{theorem}
The subscripts $p$ may be omitted from all the statements of Theorem \ref{thm:gec-properties}.
Further work concerning the conjecture can be found in \cite{MR2031413}, \cite{MR2078894}, \cite{MR2546007}
and \cite{MR3073206}.
The global epsilon constant conjecture in fact pre-dates the local epsilon constant conjecture, and the relationship
between the two conjectures  (see \cite[Theorem 4.1]{MR2078894}) is given by the equation
\begin{equation}\label{eq:lec-gec}
T\Omega^{\loc}_{p}(L/K,1) = \sum_{v} \ind^{G}_{G_{w}}(R_{L_{w}/K_{v}})
\end{equation}
where $v$ runs through all places of $K$ above $p$, $w$ is a fixed place of $L$ above $v$,
$G_{w}$ denotes the decomposition group and $\ind^{G}_{G_{w}}$ is the induction map defined in \S \ref{subsec:K-func}. Thus cases of one conjecture can often be used to establish cases of the other.
However, some care is needed because, for example, of issues at the prime $p=2$ or the fact that
 the induction map may be trivial in certain cases (see \cite[\S 4]{MR2078894}).
We now give the analogues of the results of \S \ref{subsec:lec}.

\begin{prop}
Let $A$ be a nontrivial finite abelian group of odd order and let $C_{2}$ act on $A$ by inversion.
Let $G=A \rtimes C_{2}$ as in Example \ref{ex:inversion}.
Then the $2$-part of the global epsilon constant conjecture holds
for every finite Galois extension $L/K$ of number fields with $\Gal(L/K) \simeq G$.
\end{prop}

\begin{proof}
We have $T\Omega_{p}^{\loc}(L/K,1) \in DT(\Z_{2}[G])$ by Theorem \ref{thm:gec-properties}(iii).
However, Lemma \ref{lem:trivial-DT2} shows that $DT(\Z_{2}[G])$ is trivial in this case.
\end{proof}

\begin{prop}\label{prop:global-epsilon-hybrid}
Fix a prime $p$ and let $L/K$ be a finite Galois extension of number fields with Galois group $G$.
Suppose that $\Z_{p}[G]$ is weakly $N$-hybrid.
Then the $p$-part of the global epsilon constant conjecture holds for $L/K$ if and only if it holds for $L^{N}/K$.
\end{prop}

\begin{proof}
This follows by combining Theorem \ref{thm:gec-properties}(i) and (iii) with
Proposition \ref{prop:wh-quot}.
\end{proof}

\begin{remark}
The phenomena mentioned in Remarks \ref{rmk:lec-frobenius} and \ref{rmk:lec-S4-V4} do not directly apply to Proposition \ref{prop:global-epsilon-hybrid}.
However, it is informative to consider the following example.
In Example \ref{ex:S4-V4} it was shown that $\Z_{3}[S_{4}]$ is $V_{4}$-hybrid.
Since $S_{4}/V_{4} \simeq S_{3}$ and the global epsilon constant conjecture is known for all Galois extensions $K/\Q$ with
$\Gal(K/\Q) \simeq S_{3}$ (see \cite[Theorem 1.1]{MR2031413}), Proposition \ref{prop:global-epsilon-hybrid} allows us to conclude that the $3$-part of the global epsilon constant conjecture holds for all Galois extensions $L/\Q$ with
$\Gal(L/\Q) \simeq S_{4}$. Note that in fact one can establish this result by instead using \eqref{eq:lec-gec} together with an analysis of possible decomposition groups (there are no Galois extensions $L/\Q_{3}$ with
$\Gal(L/\Q_{3}) \simeq S_{4}$ or $A_{4}$) and known cases of the local epsilon constant conjecture.
However, one advantage of using Proposition \ref{prop:global-epsilon-hybrid} is that no such analysis of possible
decomposition groups is necessary.
\end{remark}

\subsection{The leading term conjecture at $s=1$}
We examine applications to the leading term conjecture at $s=1$ formulated by Breuning and Burns in
\cite[\S 3]{MR2371375}.
Let $L/K$ be a finite Galois extension of number fields with Galois group $G$.
An element $T\Omega(L/K,1) \in K_{0}(\Z[G],\R)$ is defined and the conjecture
is that this element always vanishes.
Note that $T\Omega(L/K,1)$ should not be confused with
$T\Omega(L/K,\Z[G],1)$ used in the statement of the ETNC given in \S \ref{subsec:ETNC-prelim}.
However, the main result of \cite{MR2804251} is that the vanishing of each of these elements is equivalent if we assume Leopoldt's conjecture.

If $T\Omega(L/K,1) \in K_{0}(\Z[G],\Q)$ then the decomposition
$K_{0}(\Z[G],\Q) \simeq \oplus_{p} K_{0}(\Z_{p}[G],\Q_{p})$
(see \S \ref{subsec:p-parts}) splits $T\Omega(L/K,1)$ into $p$-parts
$T\Omega_{p}(L/K,1) \in K_{0}(\Z_{p}[G],\Q_{p})$.
We say that `$\LTC_{p}(L/K,1)$ holds' if $T\Omega_{p}(L/K,1)$ vanishes.
We give an analogue of Corollary \ref{cor:ETNC-break-down}
(for simplicity, the stated version is slightly less general than it could be).

\begin{theorem}\label{thm:ltc-s=1}
Let $L/K$ be a finite Galois extension of number fields with Galois group $G$.
Let $N=G'$ be the commutator subgroup of $G$.
Let $p$ be a prime and suppose that $\Z_{p}[G]$ is weakly $N$-hybrid.
Suppose that $L^{N}/\Q$ is abelian
(in particular, this is the case when $K=\Q$) and that $\SSC(L/K)$ holds.
Then $\LTC_{p}(L/K,1)$ holds.
\end{theorem}

\begin{proof}
Given that $\SSC(L/K)$ holds, \cite[Proposition 4.4 and Theorem 5.2]{MR2371375} show that
$T\Omega(L/K,1) \in DT(\Z[G])$. Thus $T\Omega_{p}(L/K,1) \in DT(\Z_{p}[G])$.
Furthermore $\LTC_{p}(L^{N}/K,1)$ holds by \cite[Corollary 1.3]{MR2804251}.
However,
$\quot^{G}_{G/N}(T\Omega_{p}(L/K,1))=T\Omega_{p}(L^{N}/K,1)$
by \cite[Proposition 3.5(ii)]{MR2371375} and so the result now follows from Proposition \ref{prop:wh-quot}.
\end{proof}

\begin{remark}
Using Theorem \ref{thm:ltc-s=1} and minor variants of its proof, it is straightforward to prove analogues of Theorem \ref{thm:affine-ETNC}, Corollary \ref{cor:rel-cyclic-ETNC},
and Propositions \ref{prop:ETNC-S3}, \ref{prop:ETNC-D12}, and \ref{prop:ETNC-S3-S4}, in which
the relevant cases of  $\LTC_{p}(L/K,1)$ are established.
\end{remark}

\section{Denominator Ideals}

\subsection{Generalised adjoint matrices} 
Let $G$ be a finite group and let $R$ be either $\Z$ or $\Z_{p}$ for some prime $p$.
Let $F=\Q$ or $\Q_{p}$ be the field of fractions of $R$ and
let $\mathfrak{M}(G)$ be a maximal $R$-order such that $R[G] \subseteq \mathfrak{M}(G) \subseteq F[G]$.

Choose $n \in \N$ and let $H \in M_{n \times n}(R[G])$.
Then slightly generalising the notation of \S \ref{subsec:idempotents} in an obvious way (for the case $R=\Z$),
we may decompose $H$ into
\[
H = \sum_{i=1}^{t} H_{i} \in M_{n \times n}(\mathfrak{M}(G)) = \bigoplus_{i=1}^{t}  M_{n \times n}(\mathfrak{M}(G) e_{i}),
\]
where $H_{i} := He_{i}$.
The reduced norm map $\nr:F[G] \longrightarrow \zeta(F[G])$
is defined componentwise and extends to matrix rings over $F[G]$ (see \cite[\S 7D]{MR632548}).
The reduced characteristic polynomial $f_{i}(X) = \sum_{j=0}^{n_{i} n} \alpha_{ij}X^{j}$ of $H_{i}$
has coefficients in $\mathcal{O}_{i}$.
Moreover, the constant term $\alpha_{i0}$ is equal to $\nr(H_{i}) \cdot (-1)^{n_{i} n}$.
We put
\[
H_{i}^{\ast} := (-1)^{n_{i} n +1} \cdot \sum_{j=1}^{n_{i} n} \alpha_{ij}H_{i}^{j-1}, \quad H^{\ast} := \sum_{i=1}^{t} H_{i}^{\ast}.
\]
Note that this definition of $H^{\ast}$ differs slightly from the definition in \cite[\S 4]{MR2609173},
but follows the conventions in \cite{MR3092262}.
Let $ 1_{n \times n}$ denote the $n \times n$ identity matrix.

\begin{lemma}\label{lem:ast}
We have $H^{\ast} \in M_{n\times n} (\mathfrak{M}(G))$ and $H^{\ast} H = H H^{\ast} = \nr(H) \cdot 1_{n \times n}$.
\end{lemma}

\begin{proof}
The first assertion is clear by the above considerations.
Since $f_{i}(H_{i}) = 0$, we find that
\[
H_{i}^{\ast} \cdot H_{i} = H_{i} \cdot H_{i}^{\ast}  = (-1)^{n_{i} n+1} (-\alpha_{i0}) = \nr(H_{i}),
\]
as desired.
\end{proof}

\subsection{Denominator ideals and central conductors}\label{subsec:denom-central-conductors}
We define
\begin{eqnarray*}
    \mathcal{H}(R[G]) & := & \{ x \in \zeta(R[G]) \mid xH^{\ast} \in M_{n \times n}(R[G]) \, \forall H \in M_{n \times n}(R[G]) \, \forall n \in \N \},\\
    \mathcal{I}(R[G]) & := & \langle \nr(H) \mid H \in  M_{n \times n}(R[G]), \,  n \in \N\rangle_{\zeta(R[G])}.
\end{eqnarray*}
Since $x \cdot \nr(H) = xH^{\ast}H \in \zeta(R[G])$ by Lemma \ref{lem:ast}, in particular we have
\begin{equation}\label{eq:denom-ideal}
\mathcal{H}(R[G]) \cdot \mathcal{I}(R[G]) = \mathcal{H}(R[G]) \subseteq \zeta(R[G]).
\end{equation}
Hence $\mathcal{H}(R[G])$ is an ideal in the commutative $R$-order
$\mathcal{I}(R[G])$. We will refer to $\mathcal{H}(R[G])$ as the \emph{denominator ideal} of the group ring $R[G]$.
For convenience, we put
\begin{eqnarray*}
\mathcal{I}(G) := \mathcal{I}(\Z[G]), & &  \mathcal{I}_{p}(G) := \mathcal{I}(\Z_{p}[G]), \\
\mathcal{H}(G) := \mathcal{H}(\Z[G]), & &  \mathcal{H}_{p}(G) := \mathcal{H}(\Z_{p}[G]).
\end{eqnarray*}
The following result determines the primes $p$
for which the denominator ideal $\mathcal{H}_{p}(G)$ is best possible.

\begin{prop}\label{prop:best-denominators}
We have $\mathcal{H}_{p}(G) = \zeta(\Z_{p}[G])$ if and only if $p$ does not divide the order of the commutator subgroup $G'$ of $G$. Furthermore, when this is the case $\mathcal{I}_{p}(G) = \zeta(\Z_{p}[G])$.
\end{prop}

\begin{proof}
The first assertion is a special case of \cite[Proposition 4.8]{MR3092262}, the proof of which depends heavily on
\cite[Corollary]{MR704622}. The second assertion follows from \eqref{eq:denom-ideal}.
\end{proof}

\begin{corollary}\label{cor:wh-best-hpg}
If $\Z_{p}[G]$ is weakly $G'$-hybrid then $\mathcal{H}_{p}(G) = \mathcal{I}_{p}(G)= \zeta(\Z_{p}[G])$.
\end{corollary}

\begin{proof}
This is the combination of Proposition \ref{prop:best-denominators} and
Lemma \ref{lem:weak-hybrid-p-does-not-divide-N}.
\end{proof}

\begin{remark}\label{rmk:cond-in-denom-ideal}
Let $\mathcal{F}_{p}(G)$ be the central conductor given in \eqref{eq:conductor-formula};
then $\mathcal{F}_{p}(G) \subseteq \mathcal{H}_{p}(G)$ (this follows easily from  Lemma \ref{lem:ast} and the definitions).
\end{remark}

\begin{remark}\label{rmk:interest-in-hpg}
The main reason for interest in the denominator ideal $\mathcal{H}_{p}(G)$ is that it often plays a role in annihilation results. Let $M$ be a finitely generated $\Z_{p}[G]$-module.
The `noncommutative Fitting invariant' $\Fit^{\max}_{\Z_{p}[G]}(M)$ is introduced in \cite{MR2609173}
and further developed in \cite{MR3092262}. We have the following annihilation result
(\cite[Theorem 3.6]{MR3092262})
\[
\mathcal{H}_{p}(G) \cdot \Fit^{\max}_{\Z_{p}[G]}(M) \subseteq \Ann_{\zeta(\Z_{p}[G])}(M).
\]
Furthermore, one can give annihilation results that involve $\mathcal{H}_{p}(G)$ but do not
explicitly use noncommutative Fitting invariants (see \cite[Lemma 5.1.4]{MR2845620}, for example).
\end{remark}

\begin{remark}
Let $L/K$ be a finite Galois extension of number fields and with Galois group $G$.
Fix a prime $p$ dividing $|G|$ and suppose that $\Z_{p}[G]$ is weakly $G'$-hybrid.
As we have seen in \S \ref{sec:ETNC_Tate}, we can often prove $\ETNC_{p}(L/K,r)$ for some $r \leq 0$ in this situation.
Moreover, Corollary \ref{cor:wh-best-hpg} shows that  $\mathcal{H}_{p}(G) = \zeta(\Z_{p}[G])$ is the best possible in this case. Thus the $p$-part of certain arithmetic annihilation results (such as \cite[Theorem 7.1]{MR2609173})
are both unconditional and (by Remark \ref{rmk:interest-in-hpg}) are in some sense the best possible in this situation.
In \S \ref{sec:apps-to-ann} we shall use explicit results on $\mathcal{H}_{p}(G)$ to establish sharp and explicit arithmetic annihilation results, even in certain cases where $\ETNC_{p}(L/K,r)$ is not known and
$\mathcal{H}_{p}(G)$ is strictly contained in $\zeta(\Z_{p}[G])$.
\end{remark}

\subsection{Affine transformation groups}
We compute the denominator ideals of $\Z_{p}[\Aff(q)]$, where $p$ is any prime,
$q$ is a prime power, and $\Aff(q)$ is the group of affine transformations on $\F_{q}$ defined in Example \ref{ex:affine}.

\begin{prop} \label{prop:affine-denominators}
Let $p$ be a prime and let $q = \ell^{n}$ be a prime power.
Let $\mathfrak{M}_{p}(\Aff(q))$ be a maximal $\Z_{p}$-order such that $\Z_{p}[\Aff(q)] \subseteq \mathfrak{M}_{p}(\Aff(q)) \subseteq \Q_{p}[\Aff(q)]$.
Then
    \[
        \mathcal{H}_{p}(\Aff(q)) = \left\{ \begin{array}{lll}
            \zeta(\Z_{p}[\Aff(q)]) & \mbox{ if } & p \neq \ell \\
            \mathcal{F}_{p}(\Aff(q)) & \mbox{ if } & p=\ell \neq 2;
        \end{array}\right.
    \]
    \[
        \mathcal{I}_{p}(\Aff(q)) = \left\{ \begin{array}{lll}
            \zeta(\Z_{p}[\Aff(q)]) & \mbox{ if } & p \neq \ell\\
            \zeta(\mathfrak{M}_{p}(\Aff(q))) & \mbox{ if } & p=\ell \neq 2.
        \end{array}\right.
    \]
    If $p=\ell=2$, then we have containments
    \[
        2\mathcal{H}_{2}(\Aff(q)) \subseteq \mathcal{F}_{2}(\Aff(q)) \subseteq \mathcal{H}_{2}(\Aff(q)),
    \]
    \[
        2\zeta(\mathfrak{M}_{2}(\Aff(q))) \subseteq \mathcal{I}_{2}(\Aff(q)) \subseteq \zeta(\mathfrak{M}_{2}(\Aff(q))).
    \]
\end{prop}

\begin{proof}
Let us write $G := \Aff(q) = \F_{q} \rtimes \F_{q}^{\times}$.
Then the commutator subgroup $G'$ of $G$ is isomorphic to $\F_{q}$ and hence the claim
follows from Proposition \ref{prop:best-denominators} in the case $\ell \neq p$.

Now assume $\ell=p$. We first establish the assertions regarding $\mathcal{I}_{p}(G)$.
Recall that $e_{G'} := |G'|^{-1} \sum_{\sigma \in G'} \sigma$.
Let $\chi_{nl} \in \Irr_{\C_{p}}(G)$ denote the unique
non-linear character and let $e_{nl}$ be the corresponding central idempotent.
Then $\Q_{p}[G]$ decomposes into
\[
\Q_{p}[G] =  \Q_{p}[G] e_{G'} \oplus \Q_{p}[G] e_{nl} \simeq \Q_{p}[C_{q-1}] \oplus M_{(q-1) \times (q-1)}(\Q_{p}),
\]
where we identify $G/G' \simeq C_{q-1}$.
Let $\mathfrak{M}_{p}(G)$ be a maximal $\Z_{p}$-order such that $\Z_{p}[G] \subseteq \mathfrak{M}_{p}(G) \subseteq \Q_{p}[G]$.
Then $\mathfrak{M}_{p}(G) \simeq \Z_{p}[C_{q-1}] \oplus M_{(q-1) \times (q-1)}(\Z_{p})$.
As we always have $\mathcal{I}_{p}(G) \subseteq \zeta(\mathfrak{M}_{p}(G))$ and $2$ is invertible in $\mathcal{I}_{p}(G)$
if $p \neq 2$, it suffices to show that $2\zeta(\mathfrak{M}_{p}(G)) \subseteq \mathcal{I}_{p}(G)$.
 In fact, it is sufficient to show that $2e_{G'}$ belongs to $\mathcal{I}_{p}(G)$, since in this case we also have
$2 e_{nl} = 2 - 2 e_{G'} \in \mathcal{I}_{p}(G)$ and for any $\overline{x} \in \Z_{p}[C_{q-1}]$ we may choose a lift $x \in \Z_{p}[G]$
such that $\nr(x) \cdot 2 e_{G'} = \overline{x} \cdot 2 e_{G'} \in \mathcal{I}_{p}(G)$.
If $p=\ell=2$, let $\sigma \in G'$ be an element of order $2$.
Then $(1+\sigma)(1-\sigma) = 0$, i.e., $1+\sigma \in \Z_{2}[G]$ is a zerodivisor.
In particular, $1+\sigma \not\in \Q_{2}[G]\mal$.
But $1+\sigma = 2 e_{G'} + (1+\sigma)e_{nl}$, so $(1+\sigma) e_{nl}$ is not invertible in $\Q_{2}[G] e_{nl}$.
Thus $2 e_{G'} = \nr(1+\sigma) \in \mathcal{I}_{2}(G)$ as desired.
Now assume $p=\ell>2$.
Then since $q-1$ is even we see that $\nr(-1) = - e_{G'} + (-1)^{(q-1)} e_{nl} = -e_{G'} + e_{nl}$
belongs to $\mathcal{I}_{p}(G)$. But then we also have $2 e_{G'} = 1 - \nr(-1) \in \mathcal{I}_{p}(G)$.

We now prove the assertions regarding $\mathcal{H}_{p}(G)$ in the case $\ell=p$.
By \eqref{eq:denom-ideal} we have $\mathcal{H}_{p}(G) \mathcal{I}_{p}(G) \subseteq \zeta(\Z_{p}[G])$.
If $p = \ell \neq 2$ then $\mathcal{I}_{p}(G)=\zeta(\mathfrak{M}_{p}(G))$ and so
\begin{equation}\label{eq:cc-cond-contain}
\mathcal{H}_{p}(G) \subseteq \mathcal{F}(\zeta(\mathfrak{M}_{p}(G)),\zeta(\Z_{p}[G])) :=
\{ x \in \zeta(\mathfrak{M}_{p}(G)) \mid x \zeta(\mathfrak{M}_{p}(G)) \subseteq \zeta(\Z_{p}[G] )\}.
\end{equation}
Since $\chi_{nl}(1)=q-1$ is prime to $p=\ell$,
\cite[Corollary 6.20]{MR3092262} shows that the right hand side of \eqref{eq:cc-cond-contain}
is in fact equal to $\mathcal{F}_{p}(G)$.
However, $\mathcal{F}_{p}(G) \subseteq \mathcal{H}_{p}(G)$ by Remark \ref{rmk:cond-in-denom-ideal},
and so the desired result follows.
If $p=\ell=2$ then $2\zeta(\mathfrak{M}_{2}(G)) \subseteq \mathcal{I}_{2}(G)$ and so
\eqref{eq:cc-cond-contain} holds but with $\mathcal{H}_{p}(G)$ replaced by $2\mathcal{H}_{2}(G)$,
and the desired result follows as before.
\end{proof}

\subsection{The symmetric group on four letters}

For any prime $p$ let $\mathfrak{M}_{p}(S_{4})$ be a maximal $\Z_{p}$-order
such that $\Z_{p}[S_{4}] \subseteq \mathfrak{M}_{p}(S_{4}) \subseteq \Q_{p}[S_{4}]$.
Recall that for a normal subgroup $N$ of $S_{4}$, the $N$-hybrid order $\mathfrak{M}_{p}(S_{4},N)$ was defined to be $\Z_{p}[S_{4}] e_{N} \oplus \mathfrak{M}_{p}(S_{4})(1 - e_{N})$.
Then
\begin{eqnarray*}
    \mathfrak{M}_{2}(S_{4},V_{4}) & \simeq & \Z_{2}[S_{3}] \oplus M_{3 \times 3}(\Z_{2}) \oplus M_{3 \times 3}(\Z_{2})\\
    & \simeq & \Z_{2}[C_{2}] \oplus M_{2 \times 2}(\Z_{2}) \oplus M_{3 \times 3}(\Z_{2}) \oplus M_{3 \times 3}(\Z_{2})\\
    & \simeq & \mathfrak{M}_{2}(S_{4},A_{4}),
\end{eqnarray*}
where we have used Example \ref{ex:affine} and that $S_{4} / V_{4} \simeq S_{3}$.
Let $\mathcal{F}_{2}(S_{4},A_{4})$ denote the central conductor of $\mathfrak{M}_{2}(S_{4},A_{4})$ into $\Z_{2}[S_{4}]$.

\begin{prop} \label{prop:S_4-denominators}
If $p$ is an odd prime then
$\mathcal{I}_{p}(S_{4}) = \zeta(\mathfrak{M}_{p}(S_{4}))$
and $\mathcal{H}_{p}(S_{4}) = \mathcal{F}_{p}(S_{4})$.
If $p=2$ then we have containments
\[
2\mathcal{H}_{2}(S_{4}) \subseteq \mathcal{F}_{2}(S_{4}) \subsetneqq \mathcal{F}_{2}(S_{4},A_{4}) \subseteq \mathcal{H}_{2}(S_{4}),
\]
\[
2\zeta(\mathfrak{M}_{2}(S_{4})) \subseteq \mathcal{I}_{2}(S_{4}) \subseteq
\zeta(\mathfrak{M}_{2}(S_{4},A_{4})) \subsetneqq \zeta(\mathfrak{M}_{2}(S_{4})).
\]
\end{prop}

\begin{proof}
The result is clear if $p> 3$ as $\Z_{p}[G]$ is a maximal order in this case.
If $p=3$, then
\[
\Z_{3}[S_{4}] \simeq \Z_{3}[S_{3}] \oplus M_{3 \times 3}(\Z_{3}) \oplus M_{3 \times 3}(\Z_{3}) \simeq \mathfrak{M}_{3}(S_{4},V_{4}),
\]
as we have seen in Example \ref{ex:S4-V4}.
Since $S_{3} \simeq \Aff(3)$, the result follows from Proposition \ref{prop:affine-denominators} in this case.

Now assume $p=2$.
As we have seen above, $\mathfrak{M}_{2}(S_{4},A_{4})$ is a direct sum of matrix rings over commutative rings,
hence is `nice' in the terminology of \cite{MR3092262}.
Then \cite[Proposition 4.3 and Proposition 6.2]{MR3092262} imply
that
\[
\mathcal{F}_{2}(S_{4},A_{4}) \subseteq \mathcal{H}_{2}(S_{4})
\quad \textrm{ and } \quad \mathcal{I}_{2}(S_{4}) \subseteq \zeta(\mathfrak{M}_{2}(S_{4},A_{4})).
\]
Let $e_{1}$ and $e_{2}$ be the central primitive idempotents corresponding to the two irreducible linear characters of $S_{4}$.
Suppose that $e_{1} = e_{S_{4}}$ corresponds to the trivial representation.
Let $e_{3} \in \mathfrak{M}_{2}(S_{4},A_{4})$ be the central idempotent which corresponds to the unique irreducible representation of degree $2$ (which is inflated from the standard representation of $S_{3}$).
Further let $e_{4}, e_{5} \in \mathfrak{M}_{2}(S_{4},A_{4})$ be the central idempotents which correspond to the characters of the standard representation of $S_{4}$ and of the tensor product of the standard representation and the sign representation of $S_{4}$, respectively.
Then
\[
\zeta(\mathfrak{M}_{2}(S_{4},A_{4})) \simeq \Z_{2}[C_{2}] \oplus \Z_{2} \oplus \Z_{2} \oplus \Z_{2},
\]
where the first summand corresponds to the idempotent $e_{A_{4}} = e_{1} + e_{2}$, and the other three summands have idempotents $e_{3}$, $e_{4}$ and $e_{5}$, respectively.
In order to establish $2\zeta(\mathfrak{M}_{2}(S_{4})) \subseteq \mathcal{I}_{2}(S_{4})$,
it suffices to show that $2 e_{i}$, $1\leq i \leq 5$, belong to $\mathcal{I}_{2}(S_{4})$.
As $\nr(-1) = -e_{A_{4}} + e_{3} - e_{4} - e_{5}$, we find that this is true for $2 e_{3} = 1 + \nr(-1)$.
Now let $\tau$ be a transposition.
Then a computation shows that
$\nr(\tau) = e_{1} - e_{2} - e_{3} - e_{4} + e_{5}$ and thus also
$2 (e_{1} + e_{5}) = \nr(\tau)+1$ and $2(e_{2}+e_{4}) = 1 - \nr(\tau) - 2 e_{3}$ lie in $\mathcal{I}_{2}(S_{4})$.
If $\sigma$ is a $3$-cycle, then one can compute that $\nr(1+\sigma+\sigma^{2}) = 3 e_{A_{4}}$.
But $3$ is invertible in $\Z_{2}$ and thus $e_{A_{4}} \in \mathcal{I}_{2}(S_{4})$.
Similarly, $\nr(\tau \cdot (1+\sigma+\sigma^{2})) = 3(e_{1}-e_{2})$ such that $e_{1}-e_{2}$ and
thus also $2 e_{1}$ and $2 e_{2}$ belong to $\mathcal{I}_{2}(S_{4})$.
But then also $2 e_{4} = 2(e_{2} + e_{4}) - 2 e_{2}$ and similarly $2e_{5}$ lie in $\mathcal{I}_{2}(S_{4})$.

Let $\chi_{3} \in \Irr_{\C_{2}}(S_{4})$ be the unique character of degree $2$.
Let $g$ be a lift of the $3$-cycle $(123)$ via $S_{4}/V_{4} \simeq S_{3}$.
Then $\chi_{3}(g)=-1$. Since the degrees of all other elements of $\Irr_{\C_{2}}(S_{4})$ are odd,
\cite[Proposition 6.19]{MR3092262} and Jacobinski's central conductor formula
\eqref{eq:conductor-formula} together show that
$\mathcal{F}(\zeta(\mathfrak{M}_{2}(G)),\zeta(\Z_{2}[G]))
= \mathcal{F}(\mathfrak{M}_{2}(G),\Z_{2}[G])$.
That $2\mathcal{H}_{2}(S_{4}) \subseteq \mathcal{F}_{2}(S_{4})$ now follows in the same way as
in the end of the proof of Proposition \ref{prop:affine-denominators}.
\end{proof}

\subsection{Dihedral groups of order $2\ell$}
We compute the denominator ideals of $\Z_{p}[D_{2\ell}]$, where $p$ and $\ell$ are primes and $\ell$ is odd.
In the case $\ell=3$, this result overlaps with that of Proposition \ref{prop:affine-denominators} since
$D_{6} \simeq S_{3} \simeq \Aff(3)$.

\begin{prop} \label{prop:D_2l-denominators}
Let $p$ and $\ell$ be primes with $\ell$ odd. Then
\[
\mathcal{H}_{p}(D_{2\ell}) = \left\{ \begin{array}{lll}
\zeta(\Z_{p}[D_{2\ell}]) & \mbox{ if } & p\neq \ell \\
\mathcal{F}_{p}(D_{2\ell}) & \mbox{ if } & p=\ell;
\end{array}\right.
\]
\[
\mathcal{I}_{p}(D_{2\ell}) = \left\{ \begin{array}{lll}
\zeta(\Z_{p}[D_{2\ell}]) & \mbox{ if } & p \neq \ell \\
\zeta(\mathfrak{M}_{p}(D_{2\ell})) & \mbox{ if } & p=\ell.
\end{array}\right.
\]
\end{prop}

\begin{proof}
In the case that $p \neq \ell$, the result follows from Proposition \ref{prop:best-denominators}.
In the case $p=\ell$, the result is established in \cite[Example 2.2]{MR3092262}.
\end{proof}

\section{Applications to annihilation conjectures}\label{sec:apps-to-ann}

\subsection{Equivariant $L$-values}
Let $L/K$ be a finite Galois extension of number fields with Galois group $G$.
For each place $v$ of $K$ we fix a place $w$ of $L$ above $v$ and write $G_{w}$ and $I_{w}$ for the decomposition
group and inertia subgroup of $L/K$ at $w$, respectively.
We choose a lift $\phi_{w} \in G_{w}$ of the Frobenius automorphism at $w$ and
write $\mathrm{N}(v)$ for the cardinality of the residue field of $v$.

Let $S$ be a finite set of places of $K$ containing the set $S_{\infty}$ of archimedean places.
For  $\chi \in \Irr_{\C}(G)$, we denote the $S$-truncated Artin $L$-function
attached to $\chi$ and $S$ by $L_{S}(s,\chi)$ and
we write $e_{\chi}$ for the primitive central idempotent $\chi(1)|G|^{-1}\sum_{g \in G}\chi(g^{-1})g$ of the complex group algebra $\C[G]$.
Recall that these idempotents induce a
canonical isomorphism $\zeta(\C [G]) \simeq \prod_{\chi \in \Irr_{\C} (G)} \C$.
We define the equivariant Artin $L$-function to be the meromorphic $\zeta(\C[G])$-valued function
\[
      L_{S}(s) := (L_{S}(s,\chi))_{\chi \in \Irr_{\C} (G)}.
\]

We shall henceforth assume that $S$ contains all (of the finitely many) primes that ramify in $L/K$.
Let $\nr_{\C[G]}:\C[G] \rightarrow \zeta(\C[G])$ denote the reduced norm map (see \cite[\S 7D]{MR632548})
and let $^{\#}: \C[G] \rightarrow \C[G]$ denote the involution induced by $g \mapsto g^{-1}$.
If $T$ is a second finite set of places of $K$ such that $S \cap T = \emptyset$, we define
\[
\delta_{T}(s) := \prod_{v \in T} \nr_{\C[G]}(1 - \mathrm{N}(v)^{1-s} \phi_{w}^{-1}) \quad \textrm{and} \quad \theta_{S,T}(s) := \delta_{T}(s) \cdot L_{S}(s)^{\#}.
\]
For each integer $m \geq 0$ we define an $m$-th order Stickelberger function by setting
\[
   \theta^{(m)}_{S,T}(s) := \left( \sum_{\chi \in \Irr_{\C} (G)} s^{-m \chi(1)} e_{\chi}\right) \theta_{S,T}(s).
\]
If $T$ is empty, we shall drop it from the notation.

\subsection{Regulators and class groups}\label{subsec:regs-and-classgroups}
For any finite set $S$ of places of $K$, we write $S(L)$ for the set of places of $L$ that lie above a place in $S$.
We write $\mathcal{O}_{S}$ for the ring of $S(L)$-integers in $L$.
We denote by $E_{S}$ and $\cl_{S}$ the unit group and the ideal class group of $\mathcal{O}_{S}$, respectively.
If $T$ is a second finite set of places of $K$ that is disjoint from $S$, we write $E_{S}^{T}$ for the subgroup of $E_{S}$ comprising those elements which are congruent to $1$ modulo all places in $T(L)$.
We shall always assume that $T$ is chosen so that $E_{S}^{T}$ is torsion-free;
in particular, this condition is satisfied if $T$ contains primes of two different residue characteristics or one prime of sufficiently large norm.
Moreover, we write $Y_{S}$ for the free abelian group on $S(L)$ and $X_{S}$
for the kernel of the augmentation map $Y_{S} \onto \Z$ which sends each place $w \in S(L)$ to $1$.
We let
\[
\lambda_{S}: \R \otimes_{\Z} E_{S} \rightarrow \R \otimes_{\Z} X_{S}, \quad 1 \otimes\epsilon \mapsto - \sum_{w \in S(L)} \log |\epsilon|_{w} \cdot w
\]
be the negative of the usual Dirichlet map, where $|\cdot|_{w}$ is the normalised absolute value at $w$.
For each homomorphism $\phi \in \Hom_{G}(E_{S},X_{S})$ we define a $\zeta(\R[G])$-valued regulator by
\[
    R(\phi) := \nr_{\R[G]}(\lambda_{S}^{-1} \circ (\R \otimes_{\Z} \phi)).
\]
Here we choose a finitely generated $\R[G]$-module $N$ such that $(\R \otimes_{\Z }E_{S}) \oplus N$ is a free $\R[G]$-module
and let $\nr_{\R[G]}(\lambda_{S}^{-1} \circ (\R \otimes_{\Z} \phi))$ denote the reduced norm of the matrix of
$\lambda_{S}^{-1} \circ (\R \otimes_{\Z} \phi) \oplus \mathrm{id}_{N}$ with respect to some basis of $(\R \otimes_{\Z }E_{S}) \oplus N$.
(It is easily checked that this definition is independent of the choice of $N$ and the choice of basis).

\subsection{Statement of the conjecture}
Let $0 \leq r < |S|$ be an integer and let $S_{r}$ be a subset of $S$ of cardinality $r$.
We shall say that $S$ satisfies Hypothesis $H_{r}$ if each place in $S_{r}$ splits completely in $L/K$.
If $S$ satisfies $H_{r}$, then \cite[Lemma 2.2.1]{MR2845620} implies that
$\theta_{S,T}^{(r)}(s)$ is holomorphic at $s=0$.
For each $v \in S$ we write $\Z[G/G_{w}]$ for the left $\Z[G]$-module $\Z[G] \otimes_{\Z[G_{w}]} \Z$.
We now state the following conjecture due to Burns \cite{MR2845620}.

\begin{conj}[Burns]\label{conj:Burns-conjecture}
Let $L/K$ be a Galois extension of number fields with Galois group $G$.
Assume that $S$ satisfies Hypothesis $H_{r}$ and let $S'$ be a non-empty subset of $S$ that contains both $S_{\infty}$
and $S_{r}$.
Let $a \in \mathcal{H}(G)$.
Then for any $\phi \in \Hom_{G}(E_{S},X_{S})$, we have $ \theta_{S,T}^{(r)}(0)R(\phi) \in \mathcal{I}(G)$
and therefore $a \theta_{S,T}^{(r)}(0)R(\phi)$ belongs to $\Z[G]$.
Moreover, if $r=0$ or $S'=S$, then $a\theta_{S,T}^{(r)}(0)R(\phi)$ belongs to $\Ann_{\Z[G]}(\cl_{S'})$.
In all other cases, $ba\theta_{S,T}^{(r)}(0)R(\phi)$ belongs to $\Ann_{\Z[G]}(\cl_{S'})$ for
every $b$ in $\bigcup_{v \in S \setminus S'} \Ann_{\Z[G]}(\Z[G/G_{w}])$.
\end{conj}

\begin{conj}[Weak version of Conjecture \ref{conj:Burns-conjecture}]\label{conj:weak-Burns-conjecture}
Let $\mathfrak{M}(G)$ be a maximal order such that  $\Z[G] \subseteq \mathfrak{M}(G) \subseteq \Q[G]$
and let $\mathcal{F}(G)$ denote the central conductor of $\mathfrak{M}(G)$ into $\Z[G]$.
We now obtain a weak version of Conjecture \ref{conj:Burns-conjecture} by
replacing each instance of $\mathcal{I}(G)$ and $\mathcal{H}(G)$ by $\zeta(\mathfrak{M}(G))$ and $\mathcal{F}(G)$, respectively.
\end{conj}

\begin{remark}
\hspace*{\fill}
\begin{enumerate}
\item Conjecture \ref{conj:weak-Burns-conjecture} is indeed a weakening of Conjecture \ref{conj:Burns-conjecture} since
$\mathcal{I}(G) \subseteq \zeta(\mathfrak{M}(G))$ and $\mathcal{F}(G) \subseteq \mathcal{H}(G)$ (see \S \ref{subsec:denom-central-conductors}).
\item
Note that $\theta_{S,T}^{(r)}(0)R(\phi) \in \zeta(\mathbb{Q}[G])$ if and only if Stark's conjecture holds for all $\chi \in \Irr_{\C}(G)$
such that the vanishing order of $L_{S}(s,\chi)$ at $s=0$ equals $r \chi(1)$.
In this case Conjecture \ref{conj:Burns-conjecture} naturally decomposes into $p$-parts if one replaces each instance of $\mathcal{I}(G)$, $\mathcal{H}(G)$, $\Z$ and $\cl_{S'}$ by $\mathcal{I}_{p}(G)$, $\mathcal{H}_{p}(G)$, $\Z_{p}$ and $\Z_{p} \otimes_{\Z} \cl_{S'}$, respectively.
In the same situation Conjecture \ref{conj:weak-Burns-conjecture} similarly decomposes into $p$-parts.
\item
In fact, \cite[Conjecture 2.4.1]{MR2845620} states that $\theta_{S,T}^{(r)}(0)R(\phi)$ equals the reduced norm of an integral matrix with further properties
(or  a sum of such elements if $G$ has symplectic characters).
We will not treat this issue here.
Nevertheless, in each case in which we will prove the $p$-part of Conjecture \ref{conj:Burns-conjecture} via the validity of $\ETNC_{p}(L/K,0)$, this more precise statement will also hold.
\item
Let $n$ be a positive integer.
We will say that Conjecture \ref{conj:Burns-conjecture} (or Conjecture \ref{conj:weak-Burns-conjecture})
holds up to multiplication by $n$ if all statements are true after replacing each instance of $\theta_{S,T}^{(r)}(0)$ by $n \cdot \theta_{S,T}^{(r)}(0)$.
\end{enumerate}
\end{remark}

\subsection{$\chi$-twists and annihilation}
We review some material on $\chi$-twists and annihilation results due to Burns.
Let $G$ be a finite group and let $M$ be a finitely generated (left) $\Z[G]$-module.
Then we write $M^{\vee}$ for its Pontryagin dual $\Hom(M, \Q/\Z)$ which we endow with the contragredient $G$-action.
Let $\chi \in \Irr_{\C}(G)$ and let $E_{\chi}$ be a subfield of $\C$ 
that is both Galois and of finite degree over $\Q$ and over which $\chi$ can be realised.
We put $\pr_{\chi} :=\sum_{\sigma \in G} \chi(\sigma^{-1}) \sigma=|G|\chi(1)^{-1}e_{\chi}$.
We write $\mathcal{O}_{\chi}$ for the ring of integers of $E_{\chi}$ and choose a maximal $\mathcal{O}_{\chi}$-order $\mathfrak{M}_{\chi}(G)$ in $E_{\chi}[G]$ which contains $\mathcal{O}_{\chi}[G]$.
We fix an indecomposable idempotent $f_{\chi}$ of $e_{\chi} \mathfrak{M}_{\chi}(G)$ and define an $\mathcal{O}_{\chi}$-torsionfree right $\mathcal{O}_{\chi}[G]$-module by setting $T_{\chi} := f_{\chi} \mathfrak{M}_{\chi}(G)$.
The associated right $E_{\chi}[G]$-module $V_{\chi} := E_{\chi} \otimes_{\mathcal{O}_{\chi}} T_{\chi}$ has character $\chi$ and
$T_{\chi}$ is locally free of rank $\chi(1)$ over $\mathcal{O}_{\chi}$.
For any left $\Z[G]$-module $M$ we set $M[\chi] := T_{\chi} \otimes_{\Z} M$, upon which $G$ acts on the left by
$t \otimes m \mapsto tg^{-1} \otimes g(m)$ for $t \in T_{\chi}$, $m \in M$ and $g \in G$.
We obtain a left exact functor $M \mapsto M^{\chi}$ and a right exact functor
$M \mapsto M_{\chi}$ from the category of left $\Z[G]$-modules to the category of $\mathcal{O}_{\chi}$-modules
by setting $M^{\chi} := M[\chi]^{G}$ and $M_{\chi} := M[\chi]_{G}$.

\begin{lemma}\label{lem:ann-without-extra-|G|-factor}
Let $M$ be a finitely generated (left) $\Z[G]$-module and let $\chi \in \Irr_{\C}(G)$.
Then for any $y \in \Ann_{\mathcal{O}_{\chi}}(M_{\chi})e_{\chi}$ the element
$y \pr_{\chi}$ belongs to $\mathcal{O}_{\chi} \otimes_{\Z} \Ann_{\Z[G]}(M)$.
\end{lemma}

\begin{proof}
This is \cite[Lemma 11.1.2(i)]{MR2845620}, but with a factor $|G|$ removed.
In \cite[top of p.32]{burns-p-adic-derivatives}, Burns explains how to remove this extra
factor to obtain the desired result.
We briefly explain this here for the convenience of the reader.

Let $\check{\chi} \in \Irr_{\C}(G)$ denote the character contragredient to $\chi$.
Then we can and do assume that $E_{\check{\chi}}=E_{\chi}$ and we have a natural isomorphism
\begin{equation}\label{eq:pont-twist}
(M^{\vee})^{\check{\chi}} \cong (M_{\chi})^{\vee}.
\end{equation}
We write $x \mapsto x^{ \#}$ for the $\Z$-linear involution of $\Z[G]$ that sends each element of $G$
to its inverse. Write $y'$ for the element of $\Ann_{\mathcal{O}_{\chi}}(M_{\chi})$ that is defined by $y=y'e_{\chi}$. Similarly, write $y''$ for the element of $\mathcal{O}_{\check{\chi}}$ that is defined by
$y^{ \#} = y'' e_{\check{\chi}}$. Then by \eqref{eq:pont-twist} we have
$y'' \in \Ann_{\mathcal{O}_{\check{\chi}}}((M^{\vee})^{\check{\chi}})$.
So by \cite[Lemma 11.1]{MR2771125}
(or the relevant part of the proof of \cite[Lemma 11.1.2(i)]{MR2845620}) we have
$y'' \pr_{\check{\chi}} \in \mathcal{O}_{\check\chi} \otimes_{\Z} \Ann_{\Z[G]}(M^{\vee})$.
The desired result now follows by noting that $y'' \pr_{\check{\chi}} = y^{ \#} \pr_{\check{\chi}}$
and $\Ann_{\Z[G]}(M^{\vee})=\Ann_{\Z[G]}(M)^{\#}$.
\end{proof}

\subsection{Relation to the ETNC}
We will frequently make use of the following theorem.
As is apparent from the proof, this result is due to Burns.
However, to the best of the authors' knowledge, part (ii) is not explicitly stated in the literature.

\begin{theorem}\label{thm:ETNC-implies-Burns}
Let $L/K$ be a finite Galois extension of number fields with Galois group $G$.
Let $p$ be a prime and suppose that $T\Omega(L/K, \Z[G],0) \in K_0(\Z[G], \Q)$.
\begin{enumerate}
\item If $\ETNC_p(L/K,0)$ holds, then the $p$-part of Conjecture \ref{conj:Burns-conjecture} is true.
\item If $\SSC_{p}(L/K)$ holds, then the $p$-part of Conjecture \ref{conj:weak-Burns-conjecture} is true.
\end{enumerate}
\end{theorem}

\begin{proof}
We shall refer to results which are formulated globally and not on $p$-parts. However, everything has an obvious analogue on $p$-parts: one simply has to apply the exact functor $\Z_{p} \otimes_{\Z} - $ and all proofs go through.

Part (i) is \cite[Theorem 7.2]{burns-p-adic-derivatives}, which builds on the proof of
\cite[Theorem 4.1.1]{MR2845620}.
Part (ii) is  \cite[Theorem 4.3.1 (i)]{MR2845620}, but with an extra factor $|G|$ removed.
The proof can be adapted to remove this extra factor by using
Lemma \ref{lem:ann-without-extra-|G|-factor} instead of  \cite[Lemma 11.1.2(i)]{MR2845620}.
\end{proof}

\subsection{Unconditional annihilation results}
We now put our results together to prove Conjecture \ref{conj:Burns-conjecture} in several interesting cases.
A key point is that in certain cases the $p$-part of Conjecture \ref{conj:Burns-conjecture} holds for an extension $L/K$
even when $\ETNC_{p}(L/K,0)$ is not known to hold.

\begin{theorem}
Let $q = \ell^{n}$ be a prime power and let $L/K$ be a Galois extension of number fields with $\Gal(L/K) \simeq \Aff(q)$.
Let $N$ be the Frobenius kernel of $\Aff(q)$ and suppose that $L^N/\Q$ is abelian (in particular, this is the case when $K = \Q$).
Then Conjecture \ref{conj:Burns-conjecture} holds for $L/K$
(resp.\ holds up to multiplication by $2$) if $q$ is odd (resp.\ even).
\end{theorem}

\begin{proof}
By Theorem \ref{thm:affine-ETNC} we know $\SSC(L/K)$ as well as
$\ETNC_{p}(L/K,0)$ for every prime $p \neq \ell$. Thus
we have the $p$-part of Conjecture \ref{conj:Burns-conjecture} for $p \neq \ell$ by Theorem \ref{thm:ETNC-implies-Burns}(i),
and the $\ell$-part of Conjecture \ref{conj:weak-Burns-conjecture} by Theorem \ref{thm:ETNC-implies-Burns}(ii).
However, if $\ell \neq 2$, the $\ell$-parts of the two conjectures are
equivalent by Proposition \ref{prop:affine-denominators}, and if $\ell=2$, they differ at most by a factor of $2$.
\end{proof}

%

\begin{theorem}
Let $L/K$ be a Galois extension of number fields with $\Gal(L/K) \simeq S_{3}$.
Then Conjecture \ref{conj:Burns-conjecture} is true for $L/K$.
\end{theorem}

\begin{proof}
For $p\neq3$, the $p$-part of Conjecture \ref{conj:Burns-conjecture} follows from Proposition
\ref{prop:ETNC-S3}
and Theorem \ref{thm:ETNC-implies-Burns}(i). As $\SSC(L/K)$ holds, we have the $3$-part
of Conjecture \ref{conj:weak-Burns-conjecture} by Theorem \ref{thm:ETNC-implies-Burns}(ii).
However, since $S_{3} \simeq \Aff(3)$ this is equivalent to the $3$-part of  Conjecture \ref{conj:Burns-conjecture}
by Proposition \ref{prop:affine-denominators}.
\end{proof}

\begin{theorem}\label{thm:D_2p-ann}
Let $L/\Q$ be a Galois extension with $\Gal(L/\Q) \simeq D_{2p}$ for some odd prime $p$.
Let $K/\Q$ be the unique quadratic subextension of $L/\Q$ and suppose that $K$ is imaginary.
Suppose that the class number of $K$ is a power of $p$ (possibly $1$).
Then $\SSC(L/\Q)$ holds and Conjecture \ref{conj:Burns-conjecture} holds for $L/\Q$.
\end{theorem}

\begin{proof}
Essentially the same argument as given in the proof of
Theorem \ref{thm:ETNC-dihedral}
(with subscripts `$p$' removed) shows that  $\SSC(L/\Q)$ holds.
Hence Conjecture \ref{conj:weak-Burns-conjecture} for $L/\Q$
follows from Theorem \ref{thm:ETNC-implies-Burns}(ii).
Thus it remains to establish the $\ell$-part of Conjecture \ref{conj:Burns-conjecture} for $\ell=2,p$.
Proposition \ref{prop:D_2l-denominators} shows that the $p$-parts of Conjectures
 \ref{conj:Burns-conjecture} and \ref{conj:weak-Burns-conjecture}
are in fact equivalent.
Proposition \ref{prop:ETNC2-reduction} shows that $\ETNC_{2}(L/\Q,0)$ holds and so
Theorem \ref{thm:ETNC-implies-Burns}(i) implies the $2$-part of Conjecture \ref{conj:Burns-conjecture}.
\end{proof}

\begin{remark}
Let $L/\Q$ be a Galois extension with $\Gal(L/\Q) \simeq D_{2n}$ for some odd integer $n$.
By combining Theorem \ref{thm:ETNC-dihedral} and Theorem \ref{thm:ETNC-implies-Burns}(i),
we immediately have a result showing that for certain primes $p$, the $p$-part of
Conjecture \ref{conj:Burns-conjecture} holds for $L/\Q$.
However, this gives a weaker result than Theorem \ref{thm:D_2p-ann} when specialised to the case $n=p$.
\end{remark}

\begin{theorem}\label{thm:strong-ann-S4}
Let $L/K$ be a Galois extension of number fields with $\Gal(L/K) \simeq S_{4}$.
Then Conjecture \ref{conj:weak-Burns-conjecture} is true for $L/K$ and
Conjecture \ref{conj:Burns-conjecture} holds up to multiplication by $2$.
\end{theorem}

\begin{proof}
As all characters of $S_{4}$ are rational-valued, $\SSC(L/K)$ holds by \S \ref{subsec:ETNC-max}(v).
Thus Conjecture \ref{conj:weak-Burns-conjecture} follows from Theorem \ref{thm:ETNC-implies-Burns}(ii).
However, Proposition \ref{prop:S_4-denominators} implies that
Conjecture \ref{conj:Burns-conjecture} up to multiplication by $2$ is even weaker than
Conjecture \ref{conj:weak-Burns-conjecture}.
\end{proof}

\begin{remark}
An interesting observation is that Conjectures \ref{conj:Burns-conjecture}
and \ref{conj:weak-Burns-conjecture}
in fact do not differ very much in the setting of Theorem \ref{thm:strong-ann-S4};
indeed they agree on the $3$-part.
However, by Proposition \ref{prop:S_4-denominators},
the $2$-part of Conjecture \ref{conj:Burns-conjecture} predicts more annihilators
than either Conjecture \ref{conj:weak-Burns-conjecture} or Theorem \ref{thm:strong-ann-S4},
since $\mathcal{H}_{2}(S_{4})$ strictly contains $\mathcal{F}_{2}(S_{4})$.
Note that $\ETNC(L/K, \mathfrak{M}_{2}(S_{4},A_{4}),0)$ holds because
$DT(\mathfrak{M}_{2}(S_{4},A_{4})) = 0$ and $\SSC(L/K)$ holds.
Thus it seems plausible that a proof similar to that of Theorem \ref{thm:ETNC-implies-Burns}
then implies the $2$-part of
Conjecture \ref{conj:Burns-conjecture}, but with $\mathcal{I}_{2}(S_{4})$ and $\mathcal{H}_{2}(S_{4})$
replaced by $\zeta(\mathfrak{M}_{2}(S_{4}, A_{4}))$ and $\mathcal{F}_{2}(S_{4}, A_{4})$, respectively.
\end{remark}

\begin{theorem}
Let $L/K$ be a Galois extension of number fields with $\Gal(L/K) \simeq D_{12}$.
Let $N\simeq C_{3}$ be its commutator subgroup and suppose that $L^{N}/\Q$ is abelian
(in particular, this is the case when $K = \Q$). Then Conjecture \ref{conj:Burns-conjecture} is true for $L/K$.
\end{theorem}

\begin{proof}
Proposition \ref{prop:ETNC-D12} shows that $\SSC(L/K)$ holds and $\ETNC_{p}(L/K,0)$ is true for every prime $p \neq 3$.
Note that $D_{12} \simeq S_{3} \times C_{2}$ and so $\Z_{3}[D_{12}] \simeq \Z_{3}[S_{3}] \oplus \Z_{3}[S_{3}]$; thus we
have $\mathcal{H}_{3}(D_{12}) = \mathcal{F}_{3}(D_{12})$ and $\mathcal{I}_{3}(D_{12})=\mathfrak{M}_{3}(D_{12})$ by Proposition \ref{prop:affine-denominators} (recall that $\Aff(3) \simeq S_{3}$).
Therefore the result now follows from Theorem \ref{thm:ETNC-implies-Burns}.
 \end{proof}

\subsection{The non-abelian Coates-Sinnott conjecture and generalisations}
Let $L/K$ be a finite Galois extension of number fields with Galois group $G$ and let $r<0$ be an integer.
As before,
we fix a finite set $S$ of places of $K$ containing all archimedean places and all places that ramify in $L/K$.
Let $\mathcal{J}_{r}^{S}$ be the \emph{canonical fractional Galois ideal} defined in \cite[Definition 2.9]{MR2801311}.
We will not give the rather involved definition here, because in the case of most interest to us, it may be described very easily:
Assume further that $L/K$ is an extension of totally real fields and $r<0$ is odd. Then by \cite[Remark 1(iii)]{MR2801311} we have
$\mathcal{J}_{r}^{S} = \theta_{S}(r) \cdot \zeta(\Z[\frac{1}{2}][G]) \subseteq \zeta(\Q[G]).$
In general, we have the following conjecture due to the second named author \cite[Conjecture 2.12]{MR2801311}.

\begin{conj}\label{conj:generalised-CS}
Let $L/K$ be a finite Galois extension of number fields with Galois group $G$ and let $r<0$ be an integer.
Then for every odd prime $p$ and every
$x \in \Ann_{\Z_{p}[G]}(H^{1}_{\et}(\mathcal{O}_{S}[1/p], \Z_{p}(1-r))_{\tors})$ we have
\[
\nr(x) \cdot \mathcal{H}_{p}(G) \cdot \mathcal{J}_{r}^{S} \subseteq \Ann_{\Z_{p}[G]}(H^{2}_{\et}(\mathcal{O}_{S}[1/p], \Z_{p}(1-r))),
\]
where $\nr:\Q_{p}[G] \longrightarrow \zeta(\Q_{p}[G])$ denotes the reduced norm map (see \cite[\S 7D]{MR632548}).
\end{conj}

\begin{conj}\label{conj:weak-generalised-CS}(Weak version of Conjecture \ref{conj:generalised-CS})
We obtain a weak version of Conjecture \ref{conj:generalised-CS} by replacing the denominator ideal
$\mathcal{H}_{p}(G)$ by the central conductor $\mathcal{F}_{p}(G)$.
\end{conj}

\begin{remark}
If $L$ is totally real and $r<0$ is odd,
we see by the above description of $\mathcal{J}_{r}^{S}$ that Conjecture \ref{conj:generalised-CS}
is a non-abelian generalisation of (the cohomological version of) the Coates-Sinnott conjecture \cite{MR0369322}.
Similar observations hold on minus-parts if $L/K$ is a CM-extension and $r<0$ is even.
In general, Conjecture \ref{conj:generalised-CS} generalises \cite[Conjecture 5.1]{MR2209286}.
\end{remark}

The following result is \cite[Theorem 4.1 and Theorem 5.7]{MR2801311}.

\begin{theorem} \label{thm:ETNC-implies-CS}
Let $L/K$ be a finite Galois extension of number fields with Galois group $G$.
Let $r<0$ be an integer and let $p$ be an odd prime. Suppose that $T\Omega(L/K, \Z[G],r) \in K_0(\Z[G], \Q)$. Then
\begin{enumerate}
\item If $\ETNC_{p}(L/K,r)$ holds, then Conjecture \ref{conj:generalised-CS} is true for $p$ and $r$.
\item If $\ETNC_{p}^{\max}(L/K,r)$ holds, then Conjecture \ref{conj:weak-generalised-CS} is true for $p$ and $r$.
\end{enumerate}
\end{theorem}

\begin{corollary}\label{cor:cs-tot-real}
Let $L/K$ be a finite Galois extension of totally real number fields with Galois group $G$.
Suppose that $G = N \rtimes H$ is a Frobenius group
and that $L^{N} /\Q$ is abelian (in particular, this is the case when $K=\Q$ and $H$ is abelian).
Then Conjecture \ref{conj:generalised-CS} holds for every odd $r<0$ and every prime $p \nmid 2|N|$.
If in addition $N$ is an $\ell$-group for a prime $\ell$, then Conjecture \ref{conj:generalised-CS} holds for
every odd $r<0$ and every odd prime $p$.
In particular, this applies in the following cases:
\begin{itemize}
\item $G \simeq \Aff(q)$, where $q$ is a prime power (see Example \ref{ex:affine}).
\item $G \simeq P \rtimes C_{2}$, where $P$ is an abelian $\ell$-group  (with $\ell$ odd) on which $C_{2}$ acts by inversion
    (see Example \ref{ex:inversion}).
\item $G \simeq C_{\ell} \rtimes C_{p}$, where $p<\ell$ are distinct primes such that $p \mid (\ell-1)$ and $C_{p}$ acts on $C_{\ell}$ via an embedding $C_{p} \hookrightarrow \Aut(C_{\ell})$ (see Example \ref{ex:metacyclic}).
\end{itemize}
\end{corollary}

\begin{proof}
The hypotheses are exactly those of Theorem \ref{thm:ETNC-Frobenius-groups} and so $\ETNC_{p}(L/K,r)$ holds.
Hence the result follows from Theorem \ref{thm:ETNC-implies-CS}(i).
\end{proof}

\begin{remark}
We note that in fact it is not always necessary to establish the full strength of $\ETNC_{p}(L/K,r)$ in the proof of
Corollary \ref{cor:cs-tot-real}.
For example, suppose that $G \simeq \Aff(q)$ where  $q$ is a power of an odd prime $\ell$.
In the case $p=\ell$, Proposition \ref{prop:affine-denominators} shows that
the $p$-parts of Conjectures \ref{conj:generalised-CS} and \ref{conj:weak-generalised-CS} are equivalent.
Hence by Theorem \ref{thm:ETNC-implies-CS}(ii), it suffices to establish
$\ETNC_{p}^{\max}(L/K,r)$, which does indeed hold by  \S \ref{subsec:ETNC-max}(vi), even if $L^{N}/\Q$ is non-abelian.
\end{remark}

\begin{corollary}
Let $L/K$ be a Galois extension of totally real number fields with Galois group $\Gal(L/K) \simeq S_{4}$.
Suppose that $L^{A_{4}} / \Q$ is abelian.
Then Conjecture \ref{conj:generalised-CS} holds for every odd $r<0$ and every odd prime $p$.
\end{corollary}

\begin{proof}
This follows by combining Theorem \ref{thm:ETNC-S4-negative-r} and Theorem \ref{thm:ETNC-implies-CS}(i).
\end{proof}

\begin{corollary}
Let $L/\Q$ be a Galois extension with $\Gal(L/\Q) \simeq D_{2n}$ for some odd $n$.
Let $K/\Q$ be the unique quadratic subextension of $L/\Q$ and suppose that $K$ is imaginary.
Then Conjecture \ref{conj:generalised-CS} holds for every $r<0$ and every odd prime that splits in $K/\Q$.
\end{corollary}

\begin{proof}
This follows by combining Theorem \ref{thm:ETNC-dihedral<0} and Theorem \ref{thm:ETNC-implies-CS}(i).
\end{proof}

\bibliography{hybrid-ETNC-Bib}{}
\bibliographystyle{amsalpha}

\end{document}